\documentclass[12pt,leqno]{article}

\date{}
\usepackage{amssymb,amsmath,amsthm}
\usepackage[all]{xy}

\usepackage{enumerate}
\usepackage{mathrsfs}
\usepackage{color}

\newcommand{\nc}{\newcommand}
\nc{\on}{\operatorname}

\newtheorem{theorem}{Theorem}[section]
\newtheorem{proposition}[theorem]{Proposition}
\newtheorem{lemma}[theorem]{Lemma}
\newtheorem{corollary}[theorem]{Corollary}

\theoremstyle{definition}

\newtheorem{definition}[theorem]{Definition}

\newtheorem{example}[theorem]{Example}

\newtheorem{remark}[theorem]{Remark}

\nc{\RR}{\mathrm{R}}
\nc{\LL}{\mathrm{L}}

\newcommand{\reg}{{\rm reg}}

\newcommand{\C}{{\mathbb{C}}}
\newcommand{\N}{{\mathbb{N}}}
\newcommand{\R}{{\mathbb{R}}}

\newcommand{\Z}{{\mathbb{Z}}}

\def\phi{{\varphi}}
\def\epsilon{\varepsilon}

\newcommand{\cor}{{\bf k}}

\def\shl{\mathscr{L}}

\newcommand{\symx}{{\mathfrak{X}}}

\renewcommand{\ker}{\operatorname{Ker}}

\newcommand{\rmptt}{{\{\rm pt\}}}

\renewcommand{\to}[1][]{\xrightarrow[]{#1}}

\newcommand{\isoto}[1][]{\xrightarrow[#1]%
{{\raisebox{-.6ex}[0ex][-.6ex]{$\mspace{1mu}\sim\mspace{2mu}$}}}}

\newcommand{\tto}{\rightrightarrows}
\newcommand{\To}{\mathop{\makebox[2em]{\rightarrowfill}}}

\newcommand{\muhom}{\mu hom}
\newcommand{\muHom}[1][]{\mathrm{Hom}^\mu_{\raise1.5ex\hbox to.1em{}#1}}
\newcommand{\Hom}[1][]{\mathrm{Hom}_{\raise1.5ex\hbox to.1em{}#1}}
\newcommand{\RHom}[1][]{\RR\mathrm{Hom}_{\raise1.5ex\hbox to.1em{}#1}}
\newcommand{\Ext}[2][]{\mathrm{Ext}_{\raise1.5ex\hbox to.1em{}#1}^{#2}}
\renewcommand{\hom}[1][]{{\mathscr{H}\mspace{-4mu}om}_{\raise1.5ex\hbox to.1em{}#1}}
\newcommand{\rhom}[1][]{{\RR\mathscr{H}\mspace{-3mu}om}_{\raise1.5ex\hbox to.1em{}#1}}
\newcommand{\rhomc}[1][]
{{\RR\mathscr{H}\mspace{-3mu}om}^*_{\raise1.5ex\hbox to.1em{}#1}}

\newcommand{\ext}[2][]{{\mathscr{E}xt}_{\raise1.5ex\hbox to.1em{}#1}^{#2}}
\newcommand{\Tor}[2][]{\mathrm{Tor}^{\raise1.5ex\hbox to.1em{}#1}_{#2}}
\newcommand{\tens}[1][]{\mathbin{\otimes_{\raise1.5ex\hbox to-.1em{}{#1}}}}
\newcommand{\ltens}[1][]{\mathbin{\overset{\mathrm{L}}\otimes}_{#1}}

\newcommand{\etens}{\mathbin{\boxtimes}}

\newcommand{\Endo}[1][]{\mathrm{End}_{\raise1.5ex\hbox to.1em{}#1}}

\newcommand{\Aut}[1][]{\mathrm{Aut}_{\raise1.5ex\hbox to.1em{}#1}}

\newcommand{\rsect}{\mathrm{R}\Gamma}

\newcommand{\conv}[1][]{\mathop{\circ}\limits_{#1}}

\newcommand{\eim}[1]{{#1}_!}
\newcommand{\roim}[1]{\RR{#1}_*}

\newcommand{\reim}[1]{\RR{#1}_!}
\newcommand{\opb}[1]{#1^{-1}}

\newcommand{\epb}[1]{#1^{!}}

\newcommand{\eqdot}{\mathbin{:=}}
\newcommand{\seteq}{\mathbin{:=}}

\newcommand{\cl}{\colon}
\newcommand{\scbul}{{\,\raise.4ex\hbox{$\scriptscriptstyle\bullet$}\,}}

\newcommand{\tw}[1]{\widetilde{#1}}

\newcommand{\ol}{\overline}
\newcommand{\bl}{\bigl(}
\newcommand{\br}{\bigr)}
\newcommand{\lp}{{\rm(}}
\newcommand{\rp}{{\rm)}}

\newcommand{\ba}{\begin{array}}
\newcommand{\ea}{\end{array}}

\newenvironment{anum}{
\begin{enumerate}
\itemsep=0pt

}
{\end{enumerate}}

\newcommand{\bnum}{\begin{enumerate}[{\rm(i)}]}
\newcommand{\enum}{\end{enumerate}}
\newcommand{\banum}{\begin{enumerate}[{\rm(a)}]}
\newcommand{\eanum}{\end{enumerate}}

\newcommand{\eq}{\begin{eqnarray}}
\newcommand{\eneq}{\end{eqnarray}}
\newcommand{\eqn}{\begin{eqnarray*}}
\newcommand{\eneqn}{\end{eqnarray*}}

\newcommand{\set}[2]{\left\{#1 \mathbin{;} #2 \right\}}

\nc{\Proof}{\begin{proof}}
\nc{\QED}{\end{proof}}

\def\rop{{\rm op}}

\def\dist{{\rm dist}}

\DeclareMathOperator{\id}{id}

\DeclareMathOperator{\supp}{supp}

\newcommand{\Supp}{\on{Supp}}
\newcommand{\Der}[1][]{\mathsf{D}^{#1}}
\newcommand{\Derb}{\Der[\mathrm{b}]}

\newcommand{\Derlb}{\Der[\mathrm{lb}]}

\newcommand{\SSi}{\mathrm{SS}}

\newcommand{\Int}{{\rm Int}}

\newcommand{\dT}{{\dot{T}}}
\newcommand{\dTM}{{\dT}^*M}

\newcommand{\eu}{\mathrm{eu}}
\newcommand{\DD}{\mathrm{D}}

\nc{\eps}{\varepsilon}
\nc{\hs}{\hspace*}
\nc{\nn}{\nonumber}
\nc{\tM}{\widetilde{M}}
\nc{\h}{\mathbf{h}}
\nc{\tf}{\tilde{f}}
\nc{\codim}{\on{codim}}
\nc{\lh}{H}
\nc{\bwr}{\mbox{\large{$\wr$}}}
\nc{\dTi}{\dT^{*,\mathrm{in}}}
\nc{\Cd}{\mathrm{C}}

\numberwithin{equation}{section}

\begin{document}
\title{Sheaf quantization of Hamiltonian isotopies and applications to
non-displaceability problems}
\author{St{\'e}phane Guillermou, Masaki Kashiwara and Pierre Schapira} 
\maketitle
\footnotetext[1]{Mathematics Subject Classification: 53D35, 14F05}
\footnotetext[2]{M.~K.~was partially supported by Grant-in-Aid for 
Scientific research (B) 23340005, Japan Society for the Promotion of Science}

\begin{abstract}
Let $I$ be an open interval containing $0$,
$M$ a real manifold,
$\dTM$ its cotangent bundle with the zero-section
removed and $\Phi=\{\phi_t\}_{t\in I}$ a homogeneous
Hamiltonian isotopy of $\dTM$ with $\phi_0=\id$.
Let $\Lambda\subset \dTM\times \dTM\times T^*I$ be the conic Lagrangian
submanifold associated with $\Phi$. 
We prove the existence and unicity of a sheaf $K$ on $M\times M\times I$ whose
microsupport is contained in the union of $\Lambda$ and the zero-section and
whose restriction to $t=0$ is the constant sheaf on the diagonal of  
$M\times M$.
We give applications of this result to problems of non-displaceability in
contact and symplectic topology. In particular we prove that some
strong Morse inequalities are stable by Hamiltonian isotopies and we
 also give results of non-displaceability for non-negative isotopies
in the contact setting.
\end{abstract}


\section*{Introduction}
The microlocal theory of sheaves has been introduced and systematically
developed in \cite{KS82,KS85,KS90}, the central idea being that of the
microsupport of sheaves. More precisely, consider a real manifold $M$ of class
$\Cd^\infty$ and a commutative unital ring $\cor$ of finite global dimension.
Denote by $\Derb(\cor_M)$ the bounded derived category of sheaves of
$\cor$-modules on $M$. The microsupport $\SSi(F)$ of an object $F$ of
$\Derb(\cor_M)$ is a closed subset of the cotangent bundle $T^*M$, conic for
the action of $\R^+$ on $T^*M$ and co-isotropic. Hence, this theory is
``conic'', that is, it is invariant by the $\R^+$-action and is related to the
homogeneous symplectic structure rather
than the symplectic structure of $T^*M$.

In order to treat non-homogeneous symplectic problems, a classical trick 
is to add a variable which replaces the homogeneity. This is performed for
complex symplectic manifolds in~\cite{PS04} and later in the real case 
by D.~Tamarkin in~\cite{Ta} 
who adapts the microlocal theory of sheaves to the non-homogeneous situation 
and deduces a new and very original
proof of the classical non-displaceability theorem conjectured by
Arnold. 
(Tamarkin's ideas have also been exposed and developed in~\cite{GS11}.)
Note that the use of sheaf theory in symplectic topology
already appeared in~\cite{KO01}, \cite{N06}, \cite{NZ09} and~\cite{O98}.

In this paper, we will also find a new proof of the non-displaceability 
theorem and other related
results, still remaining in the homogeneous symplectic framework, 
which makes the use of sheaf theory much easier. 
In other words, instead of adapting
microlocal sheaf theory to treat non-homogeneous geometrical problems, 
we translate these geometrical problems to homogeneous ones and apply 
the classical microlocal sheaf theory. 
Note that the converse is not always possible: there are 
interesting geometrical problems, for example those related to the notion of 
non-negative Hamiltonian isotopies,
which make sense in the homogeneous case
and which have no counterpart in the purely symplectic case.

Our main tool is, as seen in the
title of this paper, a quantization of Hamiltonian isotopies in the category
of sheaves. More precisely, we consider a homogeneous 
Hamiltonian isotopy
$\Phi=\{\phi_t\}_{t\in I}$ of $\dTM$ (the complementary of the
zero-section of $T^*M$) defined on an open interval $I$ of
$\R$ containing $0$ such that $\phi_0=\id$. 
Denoting by $\Lambda\subset \dTM\times \dTM\times T^*I$ the conic Lagrangian
submanifold associated with $\Phi$, we prove that there exists a unique 
$K\in\Der(\cor_{M\times M\times I})$ whose
microsupport is contained in the union of $\Lambda$ and the zero-section of
$T^*(M\times M\times I)$ and
whose restriction to $t=0$ is the constant sheaf on the diagonal of  $M\times M$.

We give a few applications of this result to problems of  non-displaceability
in symplectic and contact geometry.
The classical non-displaceability conjecture of Arnold 
says that, on the
cotangent bundle to a compact manifold $M$, the image of the
zero-section of $T^*M$ by a Hamiltonian isotopy always intersects 
the zero-section. 
This conjecture (and its refinements, using Morse inequalities)
have been proved by Chaperon~\cite{Cha83} who treated the case of the torus using the
methods of Conley and Zehnder~\cite{CZ83}, 
then by Hofer~\cite{Ho85} and  Laudenbach and Sikorav\cite{LS85}.
For related results in the contact case, let us quote in particular 
Chaperon~\cite{Cha95}, Chekanov~\cite{Ck96} and  Ferrand~\cite{Fe97}.

In this paper we recover the non-displaceability result in the symplectic case
as well as its refinement using Morse inequalities. Indeed, we deduce these
results from their homogeneous counterparts which are easy corollaries of our
theorem of quantization of homogeneous Hamiltonian isotopies. We also study
non-negative Hamiltonian isotopies (which make sense only in the contact
setting): we prove that given two compact connected submanifolds $X$ and $Y$
in a connected non-compact manifold $M$ and a  non-negative Hamiltonian isotopy
$\Phi=\{\phi_t\}_{t\in I}$ such that $\phi_{t_0}$ interchanges the conormal bundle
to $X$ with that of $Y$, then $X=Y$ and $\phi_t$ induces the identity on the conormal bundle
to $X$ for $t\in[0,t_0]$. The first part of these results
has already been  obtained when $X$ and $Y$ are points in~\cite{CN10,CFP10}.

\medskip
\noindent
{\bf Acknowledgments}
We have been very much stimulated by the interest of 
Claude Viterbo for the applications of sheaf theory to symplectic topology
and it is a pleasure to thank him here.

We are also extremely grateful to Emmanuel Ferrand who pointed out to us the
crucial fact that  the Arnold 
non-displaceability problem could be treated
through homogeneous symplectic methods.

\section{Microlocal theory of sheaves, after~\cite{KS90}}
\label{section:mts}
In this section, we recall some definitions and results from
\cite{KS90}, following its notations with the exception of  slight
modifications. We consider a real manifold $M$ of class $\Cd^\infty$.

\subsection{Some geometrical notions (\cite[\S 4.2,~\S 6.2]{KS90})}
In this paper we say that a $\Cd^1$-map $f\cl M\to N$ is {\em smooth} if its
differential
$d_xf\cl T_xM\to T_{f(x)}N$ is surjective for any $x\in M$.
For a locally closed subset $A$ of $M$, one denotes by $\Int(A)$
its interior and by $\overline{A}$ its closure.
One denotes by $\Delta_M$ or simply $\Delta$ the diagonal of $M\times M$. 

One denotes by $\tau_M\cl TM\to M$ and 
$\pi_M\cl T^*M\to M$ the tangent and
cotangent bundles to $M$. If $L\subset M$ is a submanifold, one
denotes by $T_LM$ its normal bundle and $T^*_LM$ its conormal
bundle. They are defined by the exact sequences 
\eqn
&&0\to TL\to L\times_MTM\to T_LM\to0,\\
&&0\to T^*_LM\to L\times_MT^*M\to T^*L\to0.
\eneqn
One sometimes identifies $M$ with the zero-section $T^*_MM$ of
$T^*M$. 
One sets $\dTM\eqdot T^*M\setminus T^*_MM$ and one denotes by
$\dot\pi_M\cl\dTM\to M$ the projection.

Let $f\cl M\to N$ be  a morphism of  real manifolds. 
To $f$ are associated the tangent morphisms
\eq\label{diag:tgmor}
&&
\ba{c}
\xymatrix{
TM\ar[d]^-{\tau_M}\ar[r]^-{f'}&M\times_NTN\ar[d]\ar[r]^-{f_\tau}&TN\ar[d]^-{\tau_N}\\
M\ar@{=}[r]&M\ar[r]^-f&N.
}\ea\eneq
By duality, we deduce the diagram: 
\eq\label{diag:cotgmor}
&&\ba{c}\xymatrix{
T^*M\ar[d]^-{\pi_M}&M\times_NT^*N\ar[d]\ar[l]_-{f_d}\ar[r]^-{f_\pi}
                                       & T^*N\ar[d]^-{\pi_N}\\
M\ar@{=}[r]&M\ar[r]^-f&N.
}\ea\eneq
One sets 
\eqn
&&T^*_MN\eqdot\ker f_d= \opb{f_d}(T^*_MM)\subset M\times_NT^*N  . 
\eneqn
Note that, denoting by $\Gamma_f$ the graph of $f$ in $M\times N$, 
the projection $T^*(M\times N)\to M\times T^*N$ identifies 
$T^*_{\Gamma_f}(M\times N)$ and $M\times_NT^*N$.

Now consider the homogeneous symplectic manifold $T^*M$: it is endowed
with the Liouville $1$-form given in a local homogeneous symplectic 
coordinate system $(x;\xi)$ on $T^*M$ by
\eqn
&&\alpha_M=\langle\xi,dx\rangle.
\eneqn
The antipodal map $a_M$ is defined by:
\eq\label{eq:antipodal}
&&a_M\cl T^*M\to T^*M,\quad(x;\xi)\mapsto(x;-\xi).
\eneq
If $A$ is a subset of $T^*M$, we denote by $A^a$ instead of $a_M(A)$ 
its image by the antipodal map.

We shall use the Hamiltonian isomorphism 
$H\cl T^*(T^*M)\isoto T(T^*M)$ given in a local symplectic coordinate system 
$(x;\xi)$ by
\eqn
&&H(\langle\lambda,dx\rangle +\langle \mu,d\xi\rangle)
=-\langle\lambda,\partial_\xi\rangle +\langle \mu,\partial_x\rangle.
\eneqn

\subsection{Microsupport (\cite[\S 5.1,~\S 6.5]{KS90})}
We consider a commutative unital ring $\cor$ of finite global dimension 
({\em e.g.} $\cor=\Z$).
(We shall assume that $\cor$ is a field
when we use Morse inequalities in section~\ref{sec:applications}.)
We denote by $\Der(\cor_M)$ (resp.\ $\Derb(\cor_M)$) 
the derived category (resp.\  bounded derived category) of sheaves
of $\cor$-modules on $M$. 
We denote by $\omega_M\in\Derb(\cor_M)$ the dualizing complex on $M$.
Recall that $\omega_M$ is isomorphic to the orientation sheaf shifted by the
dimension. 
We shall recall the definition 
of the microsupport (or singular support) $\SSi(F)$ of a sheaf $F$
(\cite[Def.~5.1.2]{KS90}).

\begin{definition}
Let $F\in \Derb(\cor_M)$ and let $p\in T^*M$. 
One says that $p\notin\SSi(F)$ if there exists an open neighborhood
$U$ of $p$ such that for any $x_0\in M$ and any
real $\Cd^1$-function $\phi$ on $M$ defined in a neighborhood of $x_0$ 
with $(x_0;d\phi(x_0))\in U$, one has
$\rsect_{\{x;\phi(x)\geq\phi(x_0)\}} (F)_{x_0}\simeq0$.
\end{definition}
In other words, $p\notin\SSi(F)$ if the sheaf $F$ has no cohomology 
supported by ``half-spaces'' whose conormals are contained in a 
neighborhood of $p$. 
\begin{itemize}
\item
By its construction, the microsupport is $\R^+$-conic, that is,
invariant by the action of  $\R^+$ on $T^*M$. 
\item
$\SSi(F)\cap T^*_MM=\pi_M(\SSi(F))=\Supp(F)$.
\item
The microsupport satisfies the triangular inequality:
if $F_1\to F_2\to F_3\to[{\;+1\;}]$ is a
distinguished triangle in  $\Derb(\cor_M)$, then 
$\SSi(F_i)\subset\SSi(F_j)\cup\SSi(F_k)$ for all $i,j,k\in\{1,2,3\}$
with $j\not=k$. 
\end{itemize}
In the sequel, for a locally closed subset $Z$ of $M$, we denote by
$\cor_Z$ the constant sheaf with stalk $\cor$ on $Z$, extended by $0$
on $M\setminus Z$.

\begin{example}\label{ex:microsupp}
(i) If $F$ is a non-zero local system on $M$ and $M$ is connected, then $\SSi(F)=T^*_MM$.

\noindent
(ii) If $N$ is a closed submanifold of $M$ and $F=\cor_N$, then 
$\SSi(F)=T^*_NM$, the conormal bundle to $N$ in $M$.

\noindent
(iii) Let $\phi$ be a $\Cd^1$-function such that $d\phi(x)\not=0$ 
whenever $\phi(x)=0$.
Let $U=\{x\in M;\phi(x)>0\}$ and let $Z=\{x\in M;\phi(x)\geq0\}$. 
Then 
\eqn
&&\SSi(\cor_U)=U\times_MT^*_MM\cup\{(x;\lambda d\phi(x));\phi(x)=0,\lambda\leq0\},\\
&&\SSi(\cor_Z)=Z\times_MT^*_MM\cup\{(x;\lambda d\phi(x));\phi(x)=0,\lambda\geq0\}.
\eneqn
\end{example}
For a precise definition of being involutive (or co-isotropic),
we refer to~\cite[Def.~6.5.1]{KS90}
\begin{theorem}
Let $F\in \Derb(\cor_M)$. Then its microsupport 
$\SSi(F)$ is involutive.
\end{theorem} 

\subsection{Localization (\cite[\S 6.1]{KS90})}
Now let $A$ be a subset of $T^*M$ and let $Z=T^*M\setminus A$. The full
subcategory $\Derb_Z(\cor_M)$ of $\Derb(\cor_M)$ consisting of sheaves $F$
such that $\SSi(F)\subset Z$ is triangulated. One sets
\eqn
&& \Derb(\cor_M;A)\eqdot \Derb(\cor_M)/\Derb_Z(\cor_M),
\eneqn
the localization of $\Derb(\cor_M)$ by $\Derb_Z(\cor_M)$. Hence, the
objects of $\Derb(\cor_M;A)$ are those of $\Derb(\cor_M)$ but a
morphism $u\cl F_1\to F_2$ in $\Derb(\cor_M)$ becomes an isomorphism
in $\Derb(\cor_M;A)$ if, after embedding  this morphism in a distinguished
triangle $F_1\to F_2\to F_3\to[+1]$, one has $\SSi(F_3)\cap A=\emptyset$. 

When $A=\{p\}$ for some $p\in T^*M$, one simply writes 
$\Derb(\cor_M;p)$ instead of $\Derb(\cor_M;\{p\})$.

\subsection{Functorial operations (\cite[\S 5.4]{KS90})}
Let $M$ and $N$ be two real manifolds. We denote by $q_i$ ($i=1,2$)
the $i$-th projection defined on $M\times N$ and by $p_i$ ($i=1,2$)
the $i$-th projection defined on $T^*(M\times N)\simeq T^*M\times T^*N$.

\begin{definition}
Let $f\cl M\to N$ be a morphism of manifolds and let 
$\Lambda\subset T^*N$ be a closed $\R^+$-conic subset. One says that 
$f$ is non-characteristic for $\Lambda$ \lp or else, $\Lambda$ 
is non-characteristic for $f$\rp\, if
\eqn
&&\opb{f_\pi}(\Lambda)\cap T^*_MN\subset M\times_NT^*_NN.
\eneqn
If $\Lambda$ is a closed $\R^+$-conic subset of $\dT^*N$, we say that
$\Lambda$ is non-characteristic for $f$ if so is $\Lambda\cup T^*_NN$.
\end{definition}
A morphism $f\cl M\to N$ is non-characteristic for a closed
$\R^+$-conic subset $\Lambda$ if and
only if $f_d\cl M\times_NT^*N\to T^*M$ is proper on $\opb{f_\pi}(\Lambda)$
and in this case 
$f_d\opb{f_\pi}(\Lambda)$ is closed and $\R^+$-conic in $T^*M$.

\begin{theorem}\label{th:opboim}\rm{(See \cite[\S 5.4]{KS90}.)}
Let $f\cl M\to N$ be a morphism of manifolds, let
$F\in\Derb(\cor_M)$ and let $G\in\Derb(\cor_N)$.
\bnum
\item
Assume that $f$ is proper on $\Supp(F)$. Then
  $\SSi(\reim{f}F)\subset f_\pi\opb{f_d}\SSi(F)$.
\item
Assume that $f$ is non-characteristic for $\SSi(G)$. Then
$\SSi(\opb{f}G)\subset f_d\opb{f_\pi}\SSi(G)$.
\item
Assume that $f$ is smooth. Then $\SSi(F)\subset M\times_NT^*N$
if and only if, for any $j\in\Z$, the sheaves $H^j(F)$
are  locally constant on the fibers of $f$.
\enum
\end{theorem}
There exist estimates of the microsupport for characteristic inverse images
and (in some special situations) for non-proper direct images but we shall not
use them here.

\begin{corollary}\label{cor:opbeqv}
Let $I$ be a contractible manifold and let $p\cl M\times I\to M$ be the projection.
If $F\in\Derb(\cor_{M\times I})$ satisfies  $\SSi(F)\subset T^*M\times T^*_II$,
then $F\simeq\opb{p}\roim{p}F$.
\end{corollary}
\begin{proof}
The result follows
from Theorem~\ref{th:opboim}~(iii) and~\cite[Prop.~2.7.8]{KS90}.
\end{proof}

\begin{corollary}\label{cor:rsectFt}
Let $I$ be an open interval of $\R$ and let $q\cl M\times I\to I$ be the projection.
Let $F\in\Derb(\cor_{M\times I})$ such that 
$\SSi(F)\cap T^*_MM\times T^*I\subset T^*_{M\times I}(M\times I)$ and 
$q$ is proper on $\Supp(F)$. Then, setting $F_a\eqdot F\vert_{\{t=a\}}$,
we have isomorphisms 
$\rsect(M; F_s) \simeq \rsect(M; F_t)$ for any $s,t\in I$.
\end{corollary}
\begin{proof}
It follows from Theorem~\ref{th:opboim} that
$\SSi(\roim{q}F)\subset T^*_II$. Therefore, there exists $M\in\Derb(\cor)$
and an isomorphism $\roim{q}F\simeq M_I$. (Recall that
$M_I=\opb{a_I}M$,
where $a_I\to\rmptt$ is the projection and $M$ is identified to a sheaf
on $\rmptt$.) 
Since $\rsect(M; F_s) \simeq (\roim{q} F)_s$, the result follows.
\end{proof}

\subsection{Morse Lemma and Morse inequalities (\cite[\S 5.4]{KS90})}
In this subsection, we consider a function $\psi\cl M\to\R$ of class $\Cd^1$. We set
\eq\label{eq:lambdaphi}
&& \Lambda_\psi=\{(x;d\psi(x))\}\subset T^*M. 
\eneq
The proposition below is a particular case of the microlocal Morse lemma
(see~\cite[Cor.~5.4.19]{KS90}) and follows from 
Theorem~\ref{th:opboim}~(ii).
The classical theory
corresponds to the constant sheaf $F=\cor_M$.
\begin{proposition}\label{prop:Morse}
Let $F\in\Derb(\cor_M)$, let $\psi\cl M\to\R$ be a function of class $\Cd^1$ and assume that $\psi$
is proper on $\Supp(F)$. Let $a<b$ in $\R$
and assume that $d\psi(x)\notin\SSi(F)$ for $a\leq \psi(x)<b$. 
For $t\in\R$, set $M_t=\opb{\psi}(]-\infty,t[)$. Then the
restriction morphism  $\rsect(M_b;F)\to\rsect(M_a;F)$ is an isomorphism.
\end{proposition}
\begin{corollary}\label{cor:Morse1}
Let $F\in\Derb(\cor_M)$ and let $\psi\cl M\to\R$ be a function of class $\Cd^1$.
Let $\Lambda_\psi$ be as in {\rm \eqref{eq:lambdaphi}}. Assume that 
\bnum
\item
$\Supp(F)$ is compact,
\item
$\rsect(M;F)\not=0$.
\enum
Then $\Lambda_\psi\cap\SSi(F)\not=\emptyset$.
\end{corollary}
Until the end of  this subsection as well as in Section~\ref{sec:sympl-case} we assume that
$\cor$ is a field.
The classical Morse inequalities are extended to sheaves
(see~\cite{ST92} and~\cite[Prop.~5.4.20]{KS90}).
Let us briefly recall this result.

For  a bounded complex $E$ of $\cor$-vector spaces 
with finite-dimensional cohomologies, we set 
\eqn
b_j(E)=\dim H^j(E),\quad b^*_l(E)=(-1)^l\sum_{j\leq l}(-1)^jb_j(E).
\eneqn
We consider a map $\psi\cl M \to \R$ of class $\Cd^1$ and define
$\Lambda_\psi$ as above. Note that we do not ask $\psi$ to be smooth. 
Let $F\in\Derb(\cor_M)$. Assume that 
\eq\label{hyp:morse3}
&&\mbox{the set }\Lambda_\psi\cap\SSi(F)\mbox{ is finite, say }\{p_1,\dots,p_N\}
\eneq
and, setting
\eq\label{eq:Vi}
&&x_i=\pi(p_i),\quad V_i\eqdot (\rsect_{\{\psi(x)\geq\psi(x_i)\}}(F))_{x_i},
\eneq
also assume that
\eq\label{hyp:morse4}
&&\mbox{ for all $i\in\{1,\dots,N\},j\in\Z$, the spaces $H^j(V_i)$ are
finite-dimensional.} 
\eneq
Set
\eqn
&&b_j(F)=b_j(\rsect(M;F)),\quad  b^*_j(F)=b^*_j(\rsect(M;F)).
\eneqn
\begin{theorem}\label{th:Morse1}
Let $F\in\Derb(\cor_M)$ and assume that $\psi$ is proper on $\Supp(F)$. 
Assume moreover~\eqref{hyp:morse3} and~\eqref{hyp:morse4}. Then
\eq\label{eq:morseineq1}
&&b_l^*(F)\leq \sum_{i=1}^N b^*_l(V_i)\mbox{ for any $l$}.
\eneq
\end{theorem}
In fact the assumption that $\psi$ is proper on $\Supp(F)$ may be weakened, see loc.\ cit.

Notice that \eqref{eq:morseineq1} immediately implies
\eq\label{eq:morseineq2}
&&b_j(F)\leq \sum_{i=1}^N b_j(V_i)\quad\text{for any $j$.}
\eneq

\subsection{Kernels (\cite[\S 3.6]{KS90})}
Let $M_i$ ($i=1,2,3$) be manifolds. For short, we write $M_{ij}\eqdot M_i\times M_j$ 
($1\leq i,j\leq3$) and $M_{123}=M_1\times M_2\times M_3$. 
We denote by $q_i$  the projection $M_{ij}\to M_i$ or the projection $M_{123}\to M_i$
and by $q_{ij}$ the projection $M_{123}\to M_{ij}$. Similarly, we denote by
$p_i$  the projection $T^*M_{ij}\to T^*M_i$ or the projection $T^*M_{123}\to T^*M_i$
and by $p_{ij}$ the projection $T^*M_{123}\to T^*M_{ij}$. We also need
to introduce the map $p_{12^a}$, the composition of $p_{12}$ and the
antipodal map on $T^*M_2$. 

Let $\Lambda_1\subset T^*M_{12}$ and $\Lambda_2\subset T^*M_{23}$. We
set 
\eq\label{eq:convolution_of_sets}
&&\Lambda_1\conv\Lambda_2\eqdot
p_{13}(\opb{{p_{12^a}}}\Lambda_1\cap\opb{p_{23}}\Lambda_2).
\eneq
We consider the operation of convolution of kernels:
\eqn
\conv[M_2]\cl\Derb(\cor_{M_{12}})\times\Derb(\cor_{M_{23}})&\to&\Derb(\cor_{M_{13}})\\
(K_1,K_2)&\mapsto&K_1\conv[M_2] K_2\eqdot
\reim{q_{13}}(\opb{q_{12}}K_1\ltens\opb{q_{23}}K_2).
\eneqn
Let  $\Lambda_i=\SSi(K_i)\subset T^*M_{i,i+1}$  and assume that 
\eq\label{eq:noncharker}
&&\left\{
\parbox{60ex}{
(i) $q_{13}$ is proper on $\opb{q_{12}}\Supp(K_1)\cap\opb{q_{23}}\Supp(K_2)$,
\\[2ex]
(ii) 
$\ba[t]{l}
\opb{p_{12^a}}\Lambda_1\cap\opb{p_{23}}\Lambda_2
\cap (T^*_{M_1}M_1\times T^*M_2\times T^*_{M_3}M_3)\\[1ex]
\hs{10ex} \subset T^*_{M_1\times M_2\times M_3}(M_1\times M_2\times M_3).
\ea$
}\right.
\eneq
It follows from Theorem~\ref{th:opboim} that under the
assumption~\eqref{eq:noncharker} we have:
\eq
&&\SSi(K_1\conv[M_2] K_2)\subset \Lambda_1\conv\Lambda_2.
\label{eq: estss}
\eneq
If there is no risk of confusion, we write $\conv$ instead of
$\conv[M_2]$. 

We will also use a relative version of the convolution of kernels.  
For a manifold $I$, $K_1 \in \Derb(\cor_{M_{12}\times I})$ and 
$K_2 \in\Derb(\cor_{M_{23}\times I})$ we set
\eq
\label{eq:rel_conv}
&K_1\conv|_I K_2\eqdot
\reim{q_{13I}}(\opb{q_{12I}}K_1\ltens\opb{q_{23I}}K_2),
\eneq
where $q_{ijI}$ is the projection $M_{123}\times I \to M_{ij}\times I$.
The above results extend to the relative case. Namely, we assume
the conditions:
\eq\label{eq:noncharkerrel}
&&\left\{
\parbox{60ex}{
(i) $\Supp(K_1)\times_{M_2\times I}\Supp(K_2)\To M_{13}\times I$ is proper,
\\[2ex]
(ii) 
$\ba[t]{l}
\opb{p_{12^aI^a}}\Lambda_1\cap\opb{p_{23I}}\Lambda_2
\cap (T^*_{M_1}M_1\times T^*M_2\times T^*_{M_3}M_3\times T^*I)\\[1ex]
\hs{20ex} \subset T^*_{M_1\times M_2\times M_3\times I}(M_1\times M_2\times M_3\times I),
\ea$
}\right.
\eneq
where 
$p_{12^aI^a}\cl T^*M_1\times T^*M_2\times T^*M_3\times T^*I\To T^*M_1\times T^*M_2\times T^*I$ 
is given by $p_{12^aI^a}(v_1,v_2,v_3,u)=(v_1,-v_2,-u)$.
Then we  have
\eq\label{eq:est_rel}
\SSi(K_1\conv|_IK_2)&\subset &
\Lambda_1\conv |_I\Lambda_2\seteq r_{13}
(r_{12^a}^{-1}\Lambda_1\cap r_{23}^{-1}\Lambda_2).
\eneq
Here, in the diagram
\eqn
&&\hs{-3ex}\xymatrix@C=-6ex@R=5ex{
&{T^*M_1\times T^*M_2\times T^*M_3\times (T^*I\times_IT^*I)}
\ar[dl]_-{r_{12^a}}\ar[d]_-{r_{13}}\ar[dr]^-{r_{23}}\\
{T^*M_1\times T^*M_2\times T^*I}&
{T^*M_1\times T^*M_3\times T^*I}&
{T^*M_2\times T^*M_3\times T^*I},
}
\eneqn
$r_{12^a}$ is given by
$p_{12^a}\cl 
T^*M_1\times T^*M_2\times T^*M_3\to T^*M_1\times T^*M_2$
and the first projection $T^*I\times_IT^*I\to T^*I$,
$r_{23}$ is given by
$p_{23}\cl 
T^*M_1\times T^*M_2\times T^*M_3\to T^*M_2\times T^*M_3$
and the second projection $T^*I\times_IT^*I\to T^*I$,
and $r_{13}$ is given by
$p_{13}\cl 
T^*M_1\times T^*M_2\times T^*M_3\to T^*M_1\times T^*M_3$
and the addition map $T^*I\times_IT^*I\to T^*I$.

\subsection{Locally bounded categories and gluing sheaves}
Although the prestack $U\mapsto\Der(\cor_U)$ ($U$ open in $M$) is not a
stack, we have the following classical result that we shall use.
\begin{lemma}\label{le:stack1}
Let $U_1$ and $U_2$ be two open subsets of $M$ and set
$U_{12}\eqdot U_1\cap U_2$.  Let $F_i\in\Der(\cor_{U_i})$ \lp$i=1,2$\rp\, and
assume we have an isomorphism
$\phi_{21}\cl F_1\vert_{U_{12}}\simeq F_2\vert_{U_{12}}$.  Then there exists
$F\in\Der(\cor_{U_1\cup U_2})$ and isomorphisms
$\phi_{i}\cl F\vert_{U_{i}}\simeq F_i$ \lp$i=1,2$\rp\, such that
$\phi_{12}=\phi_2\vert_{U_{12}}\circ\opb{\phi_1\vert_{U_{12}}}$.  Moreover,
such a triple $(F,\phi_1,\phi_2)$ is unique up to 
a \lp non-unique\rp\, isomorphism.
\end{lemma}
\begin{proof}
(i) For $* = 1,2$ or $12$ we let $j_*$ be the inclusion of $U_*$ in $U_1\cup U_2$.
By adjunction between $\eim{j_{12}}$ and $\opb{j_{12}}$, the morphism
$\phi_{21}$ defines the morphism $u\cl \eim{j_{12}} (F_1|_{U_1}) \to \eim{j_2}F_2$.
We also have the natural morphism
$v \cl \eim{j_{12}} (F_1|_{U_1}) \to \eim{j_1}F_1$.
Then $F$ is given by the distinguished triangle
\eqn
&&\eim{j_{12}}(F_1|_{U_1})\to[u\oplus v]\eim{j_2}F_2\oplus\eim{j_1}F_1\to F \to[+1].
\eneqn
\medskip\noindent
(ii) The unicity follows from the distinguished triangle
$F_{U_{12}} \to F_{U_1} \oplus F_{U_2} \to F \to[+1]$ and the fact that a
commutative square in $\Der(\cor_M)$ 
can be extended to a morphism of distinguished triangles.
\end{proof}

\begin{definition}\label{def:locbnd}
We say that $F\in \Der(\cor_M)$ is locally bounded if for any relatively
compact open subset $U\subset M$ we have $F|_U \in \Derb(\cor_U)$.  We denote
by $\Derlb(\cor_M)$ be the full subcategory of $\Der(\cor_M)$ consisting of
locally bounded objects.
\end{definition}
Local notions defined for objects of $\Derb(\cor_M)$ extend to objects of
$\Derlb(\cor_M)$, in particular the microsupport. The Grothendieck operations
which preserve boundedness properties also preserve the local boundedness
except maybe direct images.
However for $F\in \Derlb(\cor_M)$ and a morphism of manifolds $f\cl M\to N$
which is proper on $\Supp(F)$ we have
$\roim{f}F \simeq \reim{f}F \in \Derlb(\cor_N)$
and Theorem~\ref{th:opboim} still holds in this context.

In particular in the situation of the previous paragraph if we assume
that 
$K_1 \in \Derlb(\cor_{M_{12}})$ and $K_2 \in\Derlb(\cor_{M_{23}})$
satisfy \eqref{eq:noncharker}, then we
obtain $K_1\conv[M_2] K_2 \in \Derlb(\cor_{M_{13}})$ with the same bound for
its microsupport.

The category $\Derlb(\cor_M)$ is stable by the following gluing procedure.

\begin{lemma}\label{le:stack2}
Let $j_n\cl U_n \hookrightarrow M$ \lp$n\in \N$\rp\, be an increasing sequence of
open embeddings of $M$ with $\bigcup_n U_n=M$.
We consider a sequence $\{F_n\}_n$ with $F_n\in\Derlb(\cor_{U_n})$ together with isomorphisms
$u_{n+1,n}\cl F_n \isoto F_{n+1}\vert_{U_n}$.  Then there exists 
$F\in \Derlb(\cor_M)$ and isomorphisms $u_n\cl F\vert_{U_n} \simeq F_n$ 
such that $u_{n+1,n}=u_{n+1}\circ \opb{u_n}$ for all $n$.
Moreover such a family $(F,\{u_n\}_n)$ is unique up to 
a \lp non-unique\rp\ isomorphism.
\end{lemma}
\begin{proof}
(i) Denote by
$v_n\cl \eim{j_n}(F_n)\to\eim{j_{n+1}}(F_{n+1})$  the morphisms
obtained by  adjunction.
Then define $F\in\Der(\cor_M)$  as the homotopy
colimit of this system, that is, 
$F$ (which is defined up to isomorphism) is given by the distinguished triangle
\eq\label{eq:def-homot-colim}
\oplus_{n\in \N} \eim{j_n}(F_n)
\to[{v \; \eqdot \; \oplus_{n\in \N} (\id_{\eim{j_n}(F_n)} - v_n)}]
\oplus_{n\in \N} \eim{j_n}(F_n)\to F \to[+1].
\eneq
Then we have isomorphisms $u_n\cl F\vert_{U_n} \simeq F_n$ for all
$n\in\N$, $u_{n+1,n}=u_{n+1}\circ \opb{u_n}$ and $F\in \Derlb(\cor_M)$.

\medskip\noindent
(ii) Assume that we have another $G\in \Derlb(\cor_M)$ and isomorphisms
$w_n\cl G|_{U_n} \simeq F_n$.  By adjunction they give
$\phi_n\cl \eim{j_n}(F_n)\to G$ and we let $\phi$ be the sum of the
$\phi_n$'s. Since $u_{n+1,n}=w_{n+1}\circ \opb{w_n}$, we have $\phi\circ v =0$,
where $v$ is defined in~\eqref{eq:def-homot-colim}.
Hence $\phi$ factorizes through $\psi\cl F \to G$.
Then $\psi|_{U_n} = \opb{w_n}\circ u_n$ is an isomorphism.
The property of being an isomorphism being local, we obtain that $\psi$ is
an isomorphism.
\end{proof}

\subsection{Quantized contact transformations (\cite[\S 7.2]{KS90})}
Consider two manifolds $M$ and $N$, two conic open subsets $U\subset T^*M$
and $V\subset T^*N$ and a homogeneous contact transformation $\chi$:
\eq\label{eq:contact1}
&& T^*N\supset V\isoto[\chi] U\subset T^*M.
\eneq
Denote by $V^a$ the image of $V$ by the antipodal map $a_N$ on $T^*N$ and by $\Lambda$
the image of the graph of $\chi$ by $\id_{U}\times a_N$. Hence
$\Lambda$ is a conic Lagrangian submanifold of $U\times V^a$.
A quantized contact transformation (a QCT, for short) above $\chi$ is a kernel
$K\in\Derb(\cor_{M\times N})$ such that 
$\SSi(K)\cap(U\times V^a)=\Lambda$ 
and satisfying some technical
properties that we do not recall here so that the kernel $K$ induces
an equivalence of categories 
\eq\label{eq:QCT1}
&&K\conv\scbul\cl\Derb(\cor_N;V)\isoto\Derb(\cor_M;U). 
\eneq
Given $\chi$ and $q\in V$,
$p=\chi(q)\in U$, there exists such a QCT after replacing $U$ and $V$
by sufficiently small neighborhoods of $p$ and $q$. 

\subsection{The functor $\muhom$ (\cite[\S 4.4, \S 7.2]{KS90})}
The functor of microlocalization along a submanifold has been introduced by
Mikio Sato in the 70's and has been at the origin of what is now called
``microlocal analysis''.  A variant of this functor, the bifunctor
\eq
&&\muhom\cl\Derb(\cor_M)^\rop\times\Derb(\cor_M)\to\Derb(\cor_{T^*M})
\label{eq:muhom}
\eneq
has been constructed in~\cite{KS90}.
Since $\Supp(\muhom(F,F'))\subset\SSi(F)\cap\SSi(F')$, 
\eqref{eq:muhom} induces a bifunctor for any open subset $U$ of $T^*M$:
\eqn
&&\muhom\cl\Derb(\cor_M;U)^\rop\times\Derb(\cor_M;U)\to\Derb(\cor_{U}).
\eneqn
Let us only recall the properties of this functor that we shall use.
Consider a function $\psi\cl M\to\R$ defined in a neighborhood 
$W$ of $x_0\in M$ such that $d\psi(x_0)\not=0$.
Then, 
setting $S\eqdot\set{x\in W}{\psi(x)=\psi(x_0)}$ and
$p=d\psi(x_0)$, we have 
\eqn
&&\rsect_{\{\psi(x)\geq\psi(x_0)\}}(F)_{x_0}\simeq\muhom(\cor_S,F)_p
\quad\text{for any $F\in\Derb(\cor_M)$.}
\eneqn
If $\chi$ is a contact transform as in~\eqref{eq:contact1}
and if $K$ is a QCT as in~\eqref{eq:QCT1}, then $K$ induces a natural
isomorphism for any $F,G\in\Derb(\cor_N;V)$ 
\eq\label{eq:QCT2}
&&\chi_*(\muhom(F,G)\vert_V)\isoto  \muhom(K\conv F,K\conv G)\vert_U.
\eneq

\subsection{Simple sheaves (\cite[\S 7.5]{KS90})}
Let $\Lambda\subset\dTM$ be a locally closed conic Lagrangian
submanifold and let $p\in\Lambda$.
Simple sheaves along $\Lambda$ at $p$ 
are defined in~\cite[Def.~7.5.4]{KS90}. 

When $\Lambda$ is the conormal bundle to a submanifold
$N\subset M$, that is, when the projection $\pi_M\vert_\Lambda\cl\Lambda\to M$ has
constant rank, then an object 
$F\in\Derb(\cor_M)$ is simple along $\Lambda$ at
$p$ if $F\simeq\cor_N\,[d]$ in $\Derb(\cor_M;p)$ for some shift $d\in\Z$.

If $\SSi(F)$ is contained in $\Lambda$ on a neighborhood of $\Lambda$,
$\Lambda$ is connected and $F$ is simple at some point of $\Lambda$,
then $F$ is simple at every point of $\Lambda$.

If $\Lambda_1\subset T^*M_{12}$ and $\Lambda_2\subset T^*M_{23}$ are 
locally closed conic Lagrangian submanifolds and if 
$K_i\in\Derb(\cor_{M_{i,i+1}})$ ($i=1,2$) are simple along $\Lambda_i$,
then
$K_1\conv K_2$ is simple along $\Lambda_1\conv\Lambda_2$
under some conditions (see~\cite[Th.~7.5.11]{KS90}). 
In particular, simple sheaves are stable by QCT. 

Now, let $M$ and $N$ be two manifolds with the same dimension.
Let $F\in\Derb(\cor_{M\times N})$.
Set
\eq
F^{-1}=v^{-1}\rhom(F,\omega_M\etens\cor_N)\in\Derb(\cor_{N\times M}),
\label{def:invkern}
\eneq
where $v\cl N\times M\to M\times N$ is the swap.
Let $q_{ij}$ be the $(i,j)$-th projection from $N\times M\times N$.
Then we have 
$F^{-1}\conv F=\reim{q_{13}}(\opb{q_{12}}F^{-1}\ltens\opb{q_{23}}F)$.
Let $\delta\cl N\to N\times N$ be the diagonal embedding.
Then we have
$\opb{\delta}(F^{-1}\conv F)\simeq
\reim{q_2}(F\ltens \rhom(F,\omega_M\etens\cor_N))$.
Hence $\opb{\delta}(F^{-1}\conv F)\simeq
\reim{q_2}(F\ltens \rhom(F,\epb{q_2}\cor_N))
\to \reim{q_2}(\epb{q_2}\cor_N)\to\cor_N$
gives a morphism
$$F^{-1}\conv F\to\cor_{\Delta_N}.$$
\begin{proposition}%
[{\cite[Proposition 7.1.8, Proposition 7.1.9, Theorem 7.2.1]{KS90}}]%
\label{prop:invker}
Let $F\in\Derb(\cor_{M\times N})$, 
let $p_M\in\dTM$ and let $p_N\in\dT^* N$. Assume the following conditions:
\bnum
\item
$\Supp(F)\to N$ is proper,
\item $F$ is cohomologically constructible \lp see~{\rm \cite[Def.~3.4.1]{KS90}}\rp,
\item
$\SSi(F)\cap (\dTM\times T^*_NN)=\emptyset$,
\item
$\SSi(F)\cap (T^*M\times \{p_N^a\})=\{(p_M,p_N^a)\}$,
\item
$\SSi(F)$ is a Lagrangian submanifold of $T^*(M\times N)$ on a 
neighborhood of $(p_M,p_N^a)$,
\item
$F$ is simple along $\SSi(F)$
at $(p_M,p_N^a)$,
\item
$\SSi(F)\to T^*N$ is a local isomorphism 
at $(p_M,p_N^a)$. 
\enum
Then the morphism $F^{-1}\conv F\to\cor_{\Delta_N}$ is an isomorphism
in $\Derb(\cor_{N\times N};(p_N,p_N^a))$.
\end{proposition}

\section{Deformation of the conormal to the diagonal}

As usual, we denote by $\Delta_M$ or simply $\Delta$ the diagonal of $M\times M$.
We denote by $p_1$ and $p_2$ the first and second projection from 
$T^*(M\times M)$ to $T^*M$ and by $p_2^a$ the composition of $p_2$ and the
antipodal map on $T^*M$.
We also set $n\eqdot\dim M$.
 
Consider a $\Cd^\infty$-function $f(x,y)$ defined on an open neighborhood $\Omega_0
\subset M\times M$ of the diagonal $\Delta_M$.
We assume that
\bnum
\item
$f\vert_{\Delta_M}\equiv0$,
\label{cond1}
\item
$f(x,y)>0$ for $(x,y)\in \Omega_0\setminus \Delta_M$,
\item
the partial Hessian $\dfrac{\partial^2f}{\partial x_i\partial x_j}(x,y)$
is positive definite for any $(x,y)\in\Delta_M$.
\label{cond5}
\enum
Such a pair ($\Omega_0$, $f$) exists.

\begin{proposition}\label{prop:1}
Assume that \lp$\Omega_0$, $f$\rp\,  satisfies the conditions 
\eqref{cond1}--\eqref{cond5} above.
Let $U$ be a relatively compact open subset of $M$.
Then there exist an $\eps>0$ and an open subset $\Omega$ 
of $M\times M$ satisfying the following conditions:
\banum
\item
$\Delta_U\subset\Omega\subset\Omega_0\cap (M\times U)$,
\item
$Z_\eps\seteq \set{(x,y)\in\Omega}{f(x,y)\leq\eps}$ is proper over $U$
by the map induced by the second projection,
\item
for any $y\in U$ and $\eps'\in]0,\eps]$, the open subset
$\set{x\in M}{(x,y)\in\Omega,\;f(x,y)<\eps'}$ is homeomorphic to $\R^n$,
\label{cond:shere}
\item
$d_xf(x,y)\not=0$, $d_yf(x,y)\not=0$ for $(x,y)\in \Omega\setminus \Delta_M$,
\label{cond:diff}
\item
setting 
$\Gamma_{Z_\eps}=\set{(x,y;\xi,\eta)\in T^*(\Omega)}%
{f(x,y)=\eps,\;(\xi,\eta)=\lambda df(x,y),\lambda<0}$,
the projection $p_2^a\cl T^*(M\times U)\to T^*U$ 
induces an isomorphism
$\Gamma_{Z_\eps}\isoto[p_2^a]\dot T^*U$ and
the projection $p_1\cl T^*(M\times U)\to T^*M$ 
induces an open embedding
$\Gamma_{Z_\eps}\hookrightarrow\dTM$.
\eanum
\end{proposition}

\Proof
Replacing $\Omega_0$ with the open subset
\eqn
&&\Delta_M\cup\set{(x,y)\in\Omega_0}{d_xf(x,y)\not=0, d_yf(x,y)\not=0},
\eneqn
we may assume from the beginning that
$\Omega_0$ satisfies \eqref{cond:diff}.

Let $F\cl \Omega_0\to T^*M$ be the map
$(x,y)\mapsto d_yf(x,y)$.
This map sends $\Delta_M$ to $T^*_MM$ and 
is a local isomorphism.
Then there exists an open neighborhood $\Omega'\subset \Omega_0$ of 
$\Delta_M$ such that
$F\vert_{\Omega'}\cl\Omega'\to T^*M$ is an open embedding.
Hence by identifying $\Omega'$ as its image, we can reduce 
the proposition to the following lemma.
\QED

\begin{lemma}
Let $p\cl E\to X$ be a vector bundle of rank $n$,
$i\cl X\to E$ the zero-section,  $SE=(E\setminus i(X))/\R_{>0}$
the associated sphere bundle 
and $q\cl E\setminus i(X)\to SE$ the projection.
Let $f$ be a $\Cd^\infty$-function
on a neighborhood $\Omega$ of the zero-section $i(X)$ of $E$.
Assume the following conditions:
\bnum
\item
$f(z)=0$ for $z\in i(X)$,
\item
$f(z)>0$ for $z\in \Omega\setminus i(X)$,
\item for any $x\in X$ the Hessian of $f\vert_{p^{-1}(x)}$ 
at $i(x)$ is positive-definite.
\enum
Then, for any relatively compact open subset $U$ of $X$,
there exist $\eps>0$ and an open subset 
$\Omega'\subset\Omega\cap p^{-1}(U)$ containing 
$i(U)$ that satisfy the following conditions:
\anum
\item
$\set{z\in\Omega'}{f(z)\le \eps}$ is proper over $U$,
\item
$\set{z\in \Omega'}{0<f(z)<\eps}\to (SE\times_XU)\times ]0,\eps[$
given by $z\mapsto (q(z),f(z))$
is an isomorphism,
\item for any $x\in U$ and $t\in]0,\eps[$,
the set $\set{z\in\Omega'\cap p^{-1}(x)}{f(z)<t}$ is homeomorphic to $\R^n$.
\eanum
\end{lemma}
 Since the proof is elementary, we omit it.

\medskip
Recall \eqref{def:invkern}.

\begin{theorem}\label{th:2}
We keep the notations in Proposition~\ref{prop:1}
and set $L=\cor_{Z_\eps}\in\Derb(\cor_{M\times U})$. Then
$\SSi(L)\subset \Gamma_{Z_\eps}\cup Z_\eps$ and $\opb{L}\conv L\isoto\cor_{\Delta_U}$. 
\end{theorem}
\Proof
Set $Z=Z_\eps$. We have
$\SSi(\opb{L}\conv L)\subset T^*_{\Delta_U}(U\times U)\cup T^*_{U\times U}(U\times U)$.
By Proposition~\ref{prop:invker}, there exists a morphism
$\opb{L}\conv L\to \cor_{\Delta_U}$ which is an isomorphism
in $\Derb(\cor_{U\times U};\dT^*(U\times U))$.
Hence if $N\to \opb{L}\conv L\to \cor_{\Delta_U}\to[\;+1\;]$
is a distinguished triangle, then
$\SSi(N)\subset T^*_{U\times U}(U\times U)$ and hence 
$N$ has locally constant cohomologies.
In particular $\Supp(N)$ is open and closed in $U\times U$.
Let $\delta\cl U\to U\times U$ be the diagonal embedding.
Then we have
$\opb{\delta}(\opb{L}\conv L)
\simeq \reim{q_2}(L\ltens\rhom(L,\cor_{M\times U})\ltens \epb{q_2}\cor_U)$.
Since $L\simeq\cor_Z$ and
$\rhom(L,\cor_{M\times U})\simeq\cor_{\Int(Z)}$, we have
$\opb{\delta}(\opb{L}\conv L)\simeq\reim{q_2}(\cor_{\Int(Z)}\ltens
\epb{q_2}\cor_U)$.
Since the fibers of $\Int(Z)\to U$ are homeomorphic to $\R^n$,
we have
$\reim{q_2}(\cor_{\Int(Z)}\ltens
\epb{q_2}\cor_U)\simeq\cor_U$.
Thus we obtain that $\opb{\delta}(\opb{L}\conv L)\simeq \cor_U$, and hence
$\opb{\delta}N\simeq 0$.
Hence $\Supp(N)\cap\Delta_U=\emptyset$ and
$\Supp(N)\subset\Supp(\opb{L}\conv L)$.
Since we have
$$\Supp(\opb{L}\conv L)
\subset \set{(y,y')\in U\times U}
{\text{$(x,y)$, $(x,y')\in Z$ for some $x\in M$}}$$
and the fiber of $Z\to U$ is connected,
$y$ and $y'$ belong to the same connected component of $M$ as soon as 
$(y,y')\in \Supp(N)$.
Since $\Supp(N)$ is open and closed in $U\times U$ 
and $\Supp(N)\cap\Delta_U=\emptyset$,
we conclude that $\Supp(N)=\emptyset$.
\QED

\section{Quantization of homogeneous Hamiltonian isotopies}
\label{section:qisot}
Let $M$ be a real manifold of class $\Cd^\infty$ and $I$ an open interval of
$\R$ containing the origin.  
We consider a $C^\infty$-map
$\Phi\cl \dTM\times I\to \dTM$. Setting  $\phi_t=\Phi(\scbul,t)$ ($t\in I$), 
we shall always  assume 
\eq\label{hyp:isot1}
&&\begin{cases}
\mbox{$\phi_t$ is a homogeneous symplectic isomorphism for each $t\in I$,} \\
\phi_0 = \id_{\dTM}.
\end{cases}
\eneq
Let us recall here some classical facts that we will explain with 
more details in Section~\ref{sec:Hisot}.
Set 
\eqn
&& v_\Phi \eqdot \frac{\partial\Phi}{\partial t} \cl \dTM\times I\to T\dTM,\\
&&f = \langle \alpha_M, v_\Phi \rangle \cl \dTM\times I\to \R,\,f_t=f(\cdot,t).
\eneqn
Denote by $H_g$ the Hamiltonian flow of a function $g\cl\dTM\to\R$.
Then
\eqn
&&\frac{\partial\Phi}{\partial t} =H_{f_t}.
\eneqn
In other words, $\Phi$ is a homogeneous Hamiltonian isotopy.

In this situation, there exists a unique conic Lagrangian submanifold 
$\Lambda$ of $\dTM\times\dTM\times T^*I$ closed in
$\dT^*(M\times M\times I)$
such that, setting
\eq\label{eq:lambdat}
\Lambda_{t} = \Lambda\conv T^*_{t}I,
\eneq
$\Lambda_t$ is the graph of $\phi_t$. (See Lemma~\ref{lem:homog-Hamilt-isot}.)

The main result of this section is the existence and unicity of an object
$K\in \Derlb(\cor_{M\times M\times I})$ whose microsupport is contained in
$\Lambda$ outside the zero-section and whose restriction at $t=0$ is
$\cor_\Delta$.  We shall call $K$ the {\em quantization} of $\Phi$ on
$I$ or of $\{\phi_t\}_{t\in I}$.  We
first prove that if such a $K$ exists, then its support has some
properness properties from which we deduce its unicity. 
Then we prove the existence assuming 
\eq\label{hyp:isot2}
&&\left\{\parbox{60ex}{there exists  a compact subset  $A$ of $M$ such that
$\phi_t$ is the identity outside of $\opb{\dot\pi_M}(A)$ for all $t\in I$, 
${\dot\pi_M}\cl\dTM\to M$ denoting the projection.}
\right.
\eneq
 For this purpose we glue local constructions using the unicity. Then we
prove the existence
in general using an approximation of $\Phi$ by Hamiltonian isotopies
satisfying~\eqref{hyp:isot2}.

\subsection{Unicity and support of the quantization}
We introduce the notations
\eq\label{not:B}
&&\ba{rcl}
&&I_t = [0,t] \mbox{ or }[t,0] \mbox{ according to the sign of }t\in I,\\
&&B\eqdot \set{(x,y,t)\in M\times M\times I}{ (\{x\} \times \{y\} \times I_t) 
\cap \dot\pi_{M\times M\times I}(\Lambda) \not= \emptyset}.
\ea\eneq
\begin{lemma}\label{le:ttoproper}
Both projections $B\tto M\times I$ are proper.
\end{lemma}
\begin{proof}
(i) Let us show that the second projection $q\cl B\to M\times I$ is proper.
We see easily that
$\opb{q}(y,t) = \dot\pi_M( \Phi(\opb{\dot\pi_M}(y)\times I_t))
\times \{y\} \times \{t\}$.
We choose a compact set $D\subset M$ and $t\in I$.  Then
$\opb{q}(D\times I_t)$ is contained in
$\dot\pi_M( \Phi(\opb{\dot\pi_M}(D)\times I_t)) \times D \times I_t$
which is compact.

\noindent
(ii) The first projection is treated similarly by reversing 
the roles of $x$ and $y$ and replacing $\Phi$ by
$\opb\Phi =\{\opb{\phi_t} \}_{t\in I}$.
\end{proof}
Recall that for $F\in\Derlb(\cor_{M\times N})$, 
the object $\opb{F}$ is defined in~\eqref{def:invkern}. For 
an object $K\in\Derlb(\cor_{M\times M\times I})$ and $t_0\in I$, we set 
\eqn
&&K_{t_0}=K\vert_{t=t_0}\simeq K\conv\cor_{t=t_0}\in\Derlb(\cor_{M\times M}).
\eneqn
We also set (keeping the same notation for $v$ as in~\eqref{def:invkern}):
\eqn
&&K^{-1} = \opb{(v\times\id_I)}\rhom(K,\omega_M\etens\cor_M\etens\cor_I).
\eneqn
Then assuming
\eqn
&&\SSi(K)\cap T^*_{M\times M}(M\times M)\times T^*I\subset
T^*_{M\times M\times I}(M\times M \times I), 
\eneqn
we have $(K^{-1})_t \simeq (K_t)^{-1}$ for any $t\in I$. 
(See the proof of~(ii) in the proposition below.) 

\begin{proposition}\label{prop:support_unicity}
We assume that $\Phi$ satisfies hypothesis~\eqref{hyp:isot1}
and that 
$K\in\Derlb(\cor_{M\times M\times I})$ satisfies the following conditions.
\banum
\item
$\SSi(K)\subset\Lambda\cup T^*_{M\times M\times I}(M\times M\times I)$,
\label{cond:qhi1}
\item
$K_0\simeq \cor_\Delta$.
\label{cond:qhi3}
\eanum
Then we have:
\bnum
\item
$\Supp(K) \subset B$ \lp see~\eqref{not:B}\rp\,
and both projections $\Supp(K)\tto M\times I$ are proper,
\item
$K_t\conv\opb{K_t}\simeq \opb{K_t}\conv K_t\simeq \cor_\Delta$
for all $t\in I$,
\item
such a $K$ satisfying the conditions
\eqref{cond:qhi1}--\eqref{cond:qhi3} is unique up to a unique isomorphism,
\item
if there exists an open set $W\subset M$ such that
$\phi_t|_{\opb{\dot\pi_M}(W)} = \id$ for all $t\in I$, then
$K|_{(W\times M \cup M\times W)\times I}
\simeq \cor_{\Delta\times I}|_{(W\times M \cup M\times W)\times I}$.
\enum
\end{proposition}
\begin{proof}
(i) Let us prove~(i).  Since $\Lambda$ is closed and conic,
$\dot\pi_{M\times M\times I}(\Lambda)$ is closed.  So if $(x,y,t) \not\in B$
we may find open connected neighborhoods $U$ of $x$, 
$V$ of $y$ and $J$ of $I_t$ such that
$\opb{\dot\pi_{M\times M\times I}}(U\times V\times J)$ does not meet $\Lambda$.
By condition (a) this implies that $\SSi(K|_{U\times V\times J})$ is contained
in the zero-section. Hence $K$ is locally constant on $U\times V\times J$.
Now $0\in J$ and $U\times V$ does not meet $\Delta_M$ since
$\dot\pi_{M\times M\times I}(\Lambda)$ contains $\Delta_M \times \{0\}$.
Hence $K|_{U\times V\times \{0\}} = 0$ and we deduce 
$K|_{U\times V\times J} = 0$. In particular $(x,y,t) \not\in \Supp(K)$ and
this proves $\Supp(K) \subset B$. To conclude, we apply
Lemma~\ref{le:ttoproper}.

\bigskip\noindent
(ii) Let us prove~(ii). We set $F=K\conv|_I \opb{K}$
(Notation~\eqref{eq:rel_conv}). 
Hence (ii) is implied by
$F\simeq \cor_{\Delta\times I}$. 
Let $v$ be the involution of
$T^*M\times T^*M\times T^*I$
given by $v(x,\xi,x',\xi',t,\tau) = (x',-\xi',x,-\xi,t,-\tau)$.
Then we have
\eq\label{est:dual}
&&\SSi(\opb{K})\cap\dT^*(M\times M\times I)\subset v(\Lambda).
\eneq
Hence by \eqref{eq:est_rel},
$\SSi(F)$ satisfies:
\eqn
\SSi(F)&\subset&
T^*_{\Delta_M\times I}(M\times M\times I)
                 \cup T^*_{M\times M\times I}(M\times M\times I)\\
&\subset&T^*(M\times M)\times T^*_II.
\eneqn
By Corollary~\ref{cor:opbeqv}, $F$ is constant on the fibers of 
$M\times M\times I\to M\times M$.
Denote by $i_0\cl M\times M\to M\times M\times I$
the inclusion associated to $\{t=0\}\subset I$.
It is thus enough to prove the isomorphism 
$\opb{i_0}F\simeq \cor_{\Delta}$. We have
\eqn
\epb{i_0}\rhom(K,\cor_{M\times M\times I})
&\simeq&\rhom(\opb{i_0}K,\epb{i_0}\cor_{M\times M\times I})\\
&\simeq&\rhom(K_0,\cor_{M\times M})\ltens \epb{i_0}\cor_{M\times M\times I}.
\eneqn
On the other hand, the  condition on $\SSi(K)$ implies
\eqn
&&\epb{i_0}\rhom(K,\cor_{M\times M\times I})
\simeq \opb{i_0}\rhom(K,\cor_{M\times M\times I})
\ltens \epb{i_0}\cor_{M\times M\times I}.
\eneqn
Therefore
$\opb{i_0}\rhom(K,\cor_{M\times M\times I})\simeq\rhom(K_0,\cor_{M\times M})$
which gives the isomorphism $\opb{i_0}\opb{K}\simeq \opb{K_0}$.
Thus we obtain
$\opb{i_0}F\simeq K_0\conv \opb{K_0}\simeq\cor_{\Delta}$
as required.

\bigskip\noindent
(iii) is a particular case of the more precise
Lemma~\ref{lem:local-unicity} below.

\bigskip\noindent
(iv) 
We set $\tw W= (W\times M \cup M\times W)\times I$. 
Then $B\cap \tw W=\Delta_W\times I$.
Hence (i) implies that $\Supp(K)\cap \tw W\subset\Delta_W\times I$.
Then (b) implies (iv).
\end{proof}

\begin{lemma}\label{lem:local-unicity}
Let $\Phi_i\cl \dTM \times I \to \dTM$ \lp$i=1,2$\rp\, be two maps
satisfying~\eqref{hyp:isot1} and
define $\Lambda_i \subset \dTM\times\dTM\times T^*I$ as in
Lemma~\ref{lem:homog-Hamilt-isot}.
Assume that there exist
$K_i \in\Derlb(\cor_{M\times M\times I})$ \lp$i=1,2$\rp\,
satisfying conditions~\eqref{cond:qhi1}--\eqref{cond:qhi3} of
Proposition~\ref{prop:support_unicity}.
Also assume that there exists an open set $U\subset M$ 
such that
\eq\label{eq:local-unicity1}
\Phi_1|_{\opb{\dot\pi_M}(U) \times I} = \Phi_2|_{\opb{\dot\pi_M}(U) \times I}. 
\eneq
Then there exists a unique isomorphism
$\psi\cl K_1|_{M\times U \times I} \isoto K_2|_{M\times U \times I}$ 
such that
$$\xymatrix{
{\opb{i_0}K_1}\ar[rr]^{\opb{i_0}\psi}\ar[dr]^-{\sim}
&&{\opb{i_0}K_2}\ar[dl]_-{\sim}\\
&\cor_{\Delta_U}
}$$
commutes, where $i_0\cl M\times U \to M\times U\times I$ 
is the inclusion by $0\in I$.
\end{lemma}
\begin{proof}
We define $\opb{\Phi_2}\cl \dTM \times I \to \dTM$ by
$\opb{\phi_{2,t}} = (\phi_{2,t})^{-1}$ for all $t\in I$.
Then, similarly to \eqref{est:dual},
we have $\SSi(\opb{K_2}) \subset v(\Lambda_2)$ outside the zero-section.
We also define $\Phi \cl \dTM \times I \to \dTM$ by
$\phi_t = \opb{\phi_{2,t}} \circ \phi_{1,t}$.
Its associated Lagrangian submanifold is
$\Lambda = v(\Lambda_2) \conv|_I \Lambda_1$ (see \eqref{eq:est_rel}).
By~\eqref{eq:local-unicity1} we have $\phi_t(x,\xi) = (x,\xi)$ for all
$t\in I$ and $(x,\xi) \in \opb{\dot\pi_M}(U)$. Hence
\eqn
&&\Lambda\cap\dT^*(M\times U\times I)=\dT^*_{\Delta_U}(M\times U)\times T^*_II.
\eneqn
We set $L= \opb{K_2} \conv|_I K_1$.
By  \eqref{eq:est_rel},  we have the inclusion 
$\SSi(L)\subset\Lambda$ outside the zero-section. It follows that $\SSi(L)$ is contained in
$T^*(M\times U)\times T^*_II$ over $M\times U \times I$.
Since $L_0 \simeq \cor_{\Delta_M}$, we deduce from
Corollary~\ref{cor:opbeqv} that
$L|_{M\times U \times I} \simeq \cor_{( \Delta_M \cap M\times U) \times I}$.
Then
\eqn
&& K_1|_{M\times U \times I}\simeq (K_2 \conv|_I L)|_{M\times U \times I}
\simeq K_2 \conv|_I (L|_{M\times U \times I})
\simeq K_2|_{M\times U \times I}
\eneqn
as claimed.

The uniqueness of $\psi$ follows from the uniqueness of the isomorphism
$L|_{M\times U \times I} \simeq \cor_{\Delta_U \times I}$.
\end{proof}

\subsection{Existence of the quantization -- ``compact'' case}
Lemma~\ref{lem:MainLemma} below is the main step in the proof of
Theorem~\ref{th:3}.
We prove the existence of a quantization of 
a homogeneous Hamiltonian isotopy
$\Phi\cl \dTM\times I\to \dTM$ satisfying hypothesis~\eqref{hyp:isot1}
and~\eqref{hyp:isot2} (that is $\phi_t$ is the identity map outside
$\opb{\dot\pi_M}(A)$ for each $t\in I$, where $A\subset M$ is compact).
In the course of the proof we shall need an elementary lemma that we state
without proof.
\begin{lemma}\label{le:def_conormal}
Let $N$ be a manifold, $V_0\subset N$ an open subset with 
a smooth boundary, 
$C\subset N$ a compact subset and $I$ an open interval of $\R$
containing $0$.  Let $\Lambda \subset \dT^*(N\times I)$ be a closed
conic Lagrangian submanifold and  set $\Lambda_t = \Lambda \conv T^*_tI$ for
$t\in I$. We assume
\banum
\item $\Lambda_0 = \SSi(\cor_{\overline{V_0}})\cap \dT^*N$,
\item $\Lambda \cap \dT^*((N\setminus C) \times I)
 = (\Lambda_0 \cap \dT^*(N\setminus C)) \times T^*_II$,
\item $\Lambda\subset\dT^*N\times T^*I$ and 
$\Lambda\to\dT^*N\times I$ is a closed embedding.
\eanum
Then there exist $\eps>0$ with $\pm\eps\in I$ 
and an open subset $V \subset N\times ]-\eps,\eps[$
with a smooth boundary such that
\bnum
\item $V_0 = V \cap (N\times \{0\})$,
\item
$\Lambda = \SSi(\cor_{\overline V})\cap \dT^*(N\times ]-\eps,\eps[) $,
\item
$\Lambda_t = \SSi(\cor_{\overline{V \cap (N\times \{t\})}})\cap \dT^*N$
for any $t\in ]-\eps,\eps[$.
\enum
\end{lemma}

\begin{lemma}\label{lem:MainLemma}
Assume that $\Phi$ satisfies hypotheses~\eqref{hyp:isot1}
and~\eqref{hyp:isot2}.
Then there exists $K\in\Derlb(\cor_{M\times M\times I})$ satisfying
conditions~\eqref{cond:qhi1}--\eqref{cond:qhi3} of
Proposition~\ref{prop:support_unicity}.
\end{lemma}
\begin{proof}
(A) Local existence. 
We first prove that there exists $\eps>0$ such that 
there exists a quantization $\tw K\in\Derb(\cor_{M\times M\times]-\eps,\eps[}) $ 
of $\Phi$ on $]-\eps,\eps[$.

We use the results and notations of
Proposition~\ref{prop:1} and Theorem~\ref{th:2}. We choose a relatively
compact open subset $U$ such that $A\subset U\subset M$ where $A$ is given
in hypothesis~\eqref{hyp:isot2}. 
We choose $f$ and $\eps_1$ as in Proposition~\ref{prop:1}
(in which $\epsilon_1$ was denoted by $\epsilon$).  Then
$L\seteq\cor_{Z_{\eps_1}}\in\Derb(\cor_{M\times U})$ satisfies
$\SSi(L)=\Gamma_{Z_{\eps_1}}\cup Z_{\eps_1}$, and
$L^{-1}\conv L\simeq\cor_{\Delta_U}$.
We define for $t \in I$
\eqn
&&\tw\Lambda \eqdot \Gamma_{Z_{\eps_1}} \conv \Lambda
\subset \dTM\times\dT^*U\times  T^*I,\\
&&\tw\Lambda_t \eqdot \Gamma_{Z_{\eps_1}} \conv \Lambda_t
\subset \dT^*M \times \dT^*U.
\eneqn
We remark that $\tw\Lambda_t = \tw\Lambda \conv T^*_tI$ and
$\tw\Lambda_0 = \SSi(\cor_{Z_{\eps_1}})\cap\dT^*(U\times I)$.
We apply Lemma~\ref{le:def_conormal} with
$N= M\times U$, $C= A\times A$, $V_0 = \Int Z_{\eps_1}$.
We obtain $\eps>0$ and an open subset $V\subset M\times U\times ]-\eps,\eps[$
such that  $\tw L \eqdot \cor_{\ol V}
\in\Derb(\cor_{M\times U\times ]-\eps,\eps[})$ satisfies:
\banum
\item
$\SSi(\tw L)\subset (\tw\Lambda \times_I ]-\eps,\eps[) \cup 
T^*_{M\times U\times ]-\eps,\eps[}(M\times U\times ]-\eps,\eps[)$,
\item
$\tw L|_{M\times U\times \{0\}} \simeq \cor_{Z_{\eps_1}}$,
\item
the projection $M\times U\times ]-\eps,\eps[ \to U\times ]-\eps,\eps[$
is proper on $\Supp(\tw L)$.
\eanum
Now we set
\eqn
&&K=\opb{L}\conv|_I\tw L\in\Derb(\cor_{U\times U\times ]-\eps,\eps[}).
\eneqn
Then $K$ satisfies properties \eqref{cond:qhi1}--\eqref{cond:qhi3}
of Proposition~\ref{prop:support_unicity} when replacing $M$ and
 $I$ with $U$ and $]-\eps,\eps[$. 
We deduce in particular
\eqn
&&K\vert_{((U\times U)\setminus (A\times A))\times ]-\eps,\eps[}
\simeq (\cor_{\Delta_M\times ]-\eps,\eps[})
\vert_{((U\times U)\setminus (A\times A))\times ]-\eps,\eps[}.
\eneqn
Applying Lemma~\ref{le:stack1},
 $K$ extends to
$\tilde K\in \Derb(\cor_{M\times M\times ]-\eps,\eps[})$ with 
\eqn
&&\tilde K\vert_{((M\times M)\setminus (A\times A))\times ]-\eps,\eps[}
\simeq (\cor_{\Delta_M\times ]-\eps,\eps[})
\vert_{((M\times M)\setminus(A\times A))\times ]-\eps,\eps[}
\eneqn
and $\tilde K\in\Derb(\cor_{M\times M\times ]-\eps,\eps[})$ 
is a quantization of $\Phi$ on $]-\eps,\eps[$.

\medskip\noindent
(B) Gluing (a). Assume
$K^{t_0,t_1} \in \Derlb(\cor_{M\times M\times ]t_0,t_1[})$ is a quantization
of the isotopy $\{\phi_t\}_{t\in]t_0,t_1[}$ for an open interval 
$J=]t_0,t_1[\subset I$ containing the origin.

Assume that $J\not=I$.
We shall show that there exist an open interval $J'\subset I$ 
and a quantization of the
isotopy $\{\phi_t\}_{t\in J'}$ such that $J\subset J'$ and $J'\not=J$.

For an interval $I'\subset ]t_0,t_1[$, we write
$K^{t_0,t_1}|_{I'}$ for $K^{t_0,t_1}|_{M\times M\times I'}$.

Assume that $t_1\in I$.
By applying the result of (A) to the isotopy 
$\{\phi_t\circ\opb{\phi_{t_1}}\}_{t\in I}$,
there exist 
$t_0<t_3<t_1<t_4$ with $t_4\in I$ and a quantization $L^{t_3,t_4}
\in\Derb(\cor_{M\times M\times ]t_3,t_4[})$ of 
the isotopy $\{\phi_t\circ\opb{\phi_{t_1}}\}_{t\in]t_3,t_4[}$. 
Choose $t_2$ with $t_3<t_2<t_1$ and set
\eqn
F&=&(K^{t_0,t_1}\vert_{]t_3,t_1[})\conv
(K^{t_0,t_1}_{t_2})^{-1},\\
F'&=&(L^{t_3,t_4}\vert_{]t_3,t_1[})\conv
(L^{t_3,t_4}_{t_2})^{-1}.
\eneqn
Then both $F$ and $F'$ are a quantization of
the isotopy $\{\phi_t\circ\opb{\phi_{t_2}}\}_{t\in]t_3,t_1[}$.
By unicity of the quantization (Proposition~\ref{prop:support_unicity}),
$F$ and $F'$ are isomorphic and hence we have an isomorphism
\eqn
&&K^{t_3,t_4}\vert_{]t_3,t_1[}\simeq K^{t_0,t_1}\vert_{]t_3,t_1[}
\mbox{ in }\Derlb(\cor_{M\times M\times]t_3,t_1[}),
\eneqn
where
$K^{t_3,t_4}
=L^{t_3,t_4}\conv (L^{t_3,t_4}_{t_2})^{-1}\conv K^{t_0,t_1}_{t_2}
\in \Derlb(\cor_{M\times M\times ]t_3,t_4[})$.
By Lemma~\ref{le:stack1} there exists
$K^{t_0,t_4}\in\Derlb(\cor_{M\times M\times ]t_0,t_4[})$ such that
$K^{t_0,t_4}\vert_{]t_0,t_1[}\simeq K^{t_0,t_1}$ and
$K^{t_0,t_4}\vert_{]t_3,t_4[}\simeq K^{t_3,t_4}$. Then
$K^{t_0,t_4}$ is a quantization of the isotopy $\{\phi_t\}_{t\in]t_0,t_4[}$.
Similarly, if $t_0\in I$,
then there exists $t_{-1}\in I$ with $t_{-1}<t_0$ and a quantization
$K^{t_{-1},t_1}$ on $]t_{-1},t_1[$.

\medskip\noindent
(C) Gluing (b). Consider an increasing sequence of open intervals
$I_n\subset I$ and assume we have constructed quantizations $K_n$ of 
$\{\phi_t\}_{t\in I_n}$.
By unicity in Proposition~\ref{prop:support_unicity} we have
$K_{n+1}\vert_{M\times M\times I_n}\simeq K_n$. 
Set $J=\bigcup_nI_n$. By Lemma~\ref{le:stack2} there
exists $K_J\in \Derlb(\cor_{M\times M\times J})$ such that
$K_J\vert_{M\times M\times I_n}\simeq K_n$. Then $K_J$ is a quantization of
$\{\phi_t\}_{t\in J}$.
 
\medskip\noindent
(D) 
Consider the set of pairs $(J,K_J)$ where $J$ is an open interval contained in
$I$ and containing $0$ and $K_J$ is a quantization of $\{\phi_t\}_{t\in J}$.
This set, ordered by inclusion, is inductively ordered by (C). Let $(J,K_J)$
be a maximal element. It follows from (B) that $J=I$.
\end{proof}

\subsection{Existence of the quantization -- general case}
In this section we remove hypothesis~\eqref{hyp:isot2} in 
Lemma~\ref{lem:MainLemma}.
We consider $\Phi\cl \dTM \times I \to \dTM$ which only
satisfies~\eqref{hyp:isot1} and we consider $f\cl \dTM\times I\to\R$ and
$\Lambda\subset \dTM\times\dTM\times T^*I$ as above.
We will define approximations of $\Phi$ by Hamiltonian isotopies
satisfying~\eqref{hyp:isot2} such that their quantizations ``stabilize''
over compact sets which allows us to define the quantization of $\Phi$.

We consider a $\Cd^\infty$-function $g\cl M\to \R$ with compact support. The
function
$$
f_g \cl \dTM\times I \to\R,
\qquad
(x,\xi,t)  \mapsto g(x) f(x,\xi,t)
$$
is homogeneous of degree $1$ and has support in
$\opb{\dot\pi_M}(A) \times I$ with $A = \supp(g)$ compact.  So its Hamiltonian
flow is well-defined and satisfies~\eqref{hyp:isot1} and~\eqref{hyp:isot2}.
We denote it by
$$
\Phi_g \cl \dTM\times I\to \dTM
$$
and we let $\Lambda_g \subset \dTM \times \dTM \times T^*I$ be the
Lagrangian submanifold associated to $\Phi_g$ in
Lemma~\ref{lem:homog-Hamilt-isot}. By Lemma~\ref{lem:MainLemma} there exists
a unique
$K_g \in \Derlb(\cor_{M\times M\times I})$ such that
$$
\SSi(K_g) \subset \Lambda_g \cup T^*_{M\times M\times I}(M\times M\times I)
\quad\mbox{and}\quad
K_g|_{M\times M\times \{0\}} \simeq \cor_{\Delta_M}.
$$
\begin{lemma}\label{lem:approx-Phi}
Let $U\subset M$ be a relatively compact open subset and $J\subset I$ a
relatively compact open subinterval. Then there exists a $\Cd^\infty$-function
$g\cl M\to \R$ with compact support such that
\eq\label{eq:approx-Phi1}
\Phi_g|_{\dot\pi_M(U) \times J} = \Phi|_{\dot\pi_M(U) \times J}. 
\eneq
\end{lemma}
\begin{proof}
We assume without loss of generality that $0\in J$.
Since $\Phi$ is homogeneous, the set
$C \eqdot
\dot\pi_M( \Phi( \opb{\dot\pi_M}(\overline U) \times \overline J) )$
is a compact subset of $M$.  We choose a $\Cd^\infty$-function $g\cl M\to \R$
with compact support such that $g$ is equal to the constant function
$1$ on some neighborhood of $C$.

Then for any $p \in \opb{\dot\pi_M}(U)$ the functions $f$ and $f_g$ coincide
on a neighborhood of
$\gamma_{p,J} \eqdot \{(\phi_t(p),t); \;t\in J\} \subset \dTM\times I$ which
is the trajectory of $p$ by the flow $\Phi$.  Hence their Hamiltonian vector
fields coincide on $\gamma_{p,J}$ and their flows also.
\end{proof}

\begin{theorem}\label{th:3}
We consider $\Phi\cl \dTM\times I\to \dTM$ and we assume that it satisfies
hypothesis~\eqref{hyp:isot1}.
Then there exists $K\in\Derlb(\cor_{M\times M\times I})$ satisfying
conditions~\eqref{cond:qhi1}--\eqref{cond:qhi3} of
Proposition~\ref{prop:support_unicity}.
\end{theorem}
\begin{proof}
We consider an increasing sequence of relatively compact open subsets
$U_n \subset M$, $n\in \N$, and an increasing sequence of relatively compact
open subintervals $J_n\subset I$, $n\in \N$, such that
$\bigcup_{n\in \N} U_n = M$ and $\bigcup_{n\in \N} J_n = I$.
By Lemma~\ref{lem:approx-Phi} we can choose $\Cd^\infty$-functions
$g_n\cl M\to \R$ with compact supports such that $\Phi_{g_n}$ and $\Phi$
coincide on $\dot\pi_M(U_n) \times J_n$.
We let $K_n \in \Derlb(\cor_{M\times M\times I})$ be the quantization of
$\Phi_{g_n}$. In particular
\eq
\label{eq:SSKn}
&&\ba{l}
 \SSi(K_n) \cap \dT^*(M\times U_n \times J_n)
 \subset \Lambda \cap \dT^*(M\times U_n \times J_n),  \\[.5ex]
K_n|_{M\times M\times \{0\}} \simeq \cor_{\Delta_M}.
\ea
\eneq
By Lemma~\ref{lem:local-unicity} we have isomorphisms
$K_{n+1}|_{M\times U_n \times J_n}\simeq K_n|_{M\times U_n \times
  J_n}$.
Then Lemma~\ref{le:stack2} implies that
there exists $K\in \Derlb(\cor_{M\times M\times I})$ 
with isomorphisms 
$K|_{M\times U_n \times J_n}\simeq K_n|_{M\times U_n \times J_n}$
for any $n\in \N$.
Then $K$ satisfies~\eqref{cond:qhi1}--\eqref{cond:qhi3} of
Proposition~\ref{prop:support_unicity} by \eqref{eq:SSKn}.
\end{proof}
\begin{remark}\label{rem:compact-case-bounded}
If $\Phi$ also satisfies~\eqref{hyp:isot2} and $J$ is a relatively compact
subinterval of $I$, then the restriction $K|_{M\times M\times J}$ belongs to
$\Derb(\cor_{M\times M\times J})$. See Example~\ref{exam:counter}.
\end{remark}
\begin{remark}\label{rem:contractible}
Theorem~\ref{th:3} extends immediately when replacing the open
interval $I$ with a smooth contractible manifold $U$ with a marked
point $u_0$.
Indeed, consider a homogeneous Hamiltonian isotopy
$\Phi\cl \dT^*M \times U \to \dT^*M$ with $\phi_{u_0}=\id_{\dT^*M}$. 
We can construct a Lagrangian submanifold 
$\Lambda$ of $\dT^*(M\times M\times U)$ as in Lemma~\ref{lem:homog-Hamilt-isot}.
Set $\delta=\id_{M\times M}\times\delta_U$ where 
$\delta_U\cl U\to U\times U$ is the diagonal embedding.
One easily sees that
$\delta_\pi \opb{\delta_d} (\Lambda)$ is the graph of a homogeneous symplectic
diffeomorphism $\tw \Phi$ of $\dT^*(M\times U)$.

Now let $h\cl U\times I \to U$ be a retraction
to $u_0$ such that $h(U\times \{0\}) = \{u_0\}$ and $h|_{U\times \{1\}} =\id_U$.
We set $\Phi_t = \Phi \circ (\id_{\dT^*M} \times h_t)$ and we apply the above
procedure to each $\Phi_t$.
Then $\{\tw\Phi_t\}_{t\in I}$ is a homogeneous Hamiltonian isotopy of
$\dT^*(M\times U)$ with $\tw\Phi_0 = \id$ and
Theorem~\ref{th:3} associates with it a kernel
$K\in \Derlb(\cor_{M\times U \times M\times U \times I})$.
One checks that $K$ is supported on $M\times M\times \Delta_U \times I$
and that $K|_{M\times M\times \Delta_U \times \{1\}}$ is the desired kernel.
\end{remark}

\begin{example}
Let $M=\R^n$ and denote by $(x;\xi)$ the homogeneous symplectic
coordinates on $T^*\R^n$. Consider the isotopy 
$\phi_t(x;\xi)= (x-t\frac{\xi}{\vert\xi\vert};\xi)$, $t\in I=\R$.
Then 
\eqn
&&\Lambda_t=\{(x,y,\xi,\eta);
\vert x-y\vert=\vert t\vert,\;\xi=-\eta=s(x-y),\; st<0\}
\quad\mbox{ for }t\neq0,\\
&&\Lambda_0=\dot T^*_\Delta (M\times M).
\eneqn
The isomorphisms 
\eqn
&&\rhom(\cor_{\Delta\times\{t=0\}},\cor_{M\times M\times\R})
\simeq\cor_{\Delta\times\{t=0\}}[-n-1]\\
&&\rhom(\cor_{\{\vert x-y\vert\leq-t\}},\cor_{M\times M\times\R})
\simeq \cor_{\{\vert x-y\vert<-t\}}
\eneqn
together with the morphism
$\cor_{\{\vert x-y\vert\leq-t\}}\to\cor_{\Delta\times\{t=0\}}$
induce the morphism
$\cor_{\Delta\times\{t=0\}}[-n-1]\to \cor_{\{\vert x-y\vert<-t\}}$.
Hence we obtain
\eqn
&&\cor_{\{\vert x-y\vert\leq t\}}
\to \cor_{\Delta\times\{t=0\}}
\to\cor_{\{\vert x-y\vert<-t\}}[n+1].
\eneqn
Let $\psi$ be the composition.
Then there exists a distinguished triangle in $\Derb(\cor_{M\times M\times I})$:
\eqn
&&\cor_{\{\vert x-y\vert<-t\}}[n]\to K\to 
\cor_{\{\vert x-y\vert\leq t\}}\xrightarrow[{\ \psi\ }]{+1}
\eneqn
We can verify that $K$ satisfies the
properties~\eqref{cond:qhi1}--\eqref{cond:qhi3} of
Proposition~\ref{prop:support_unicity}.
{}From this distinguished triangle, we deduce the isomorphisms 
in $\Derb(\cor_{M\times M})$:
$K_t\simeq \cor_{\{\vert x-y\vert\leq t\}}$ for $t\geq0$
and  $K_t\simeq\cor_{\{\vert x-y\vert<-t\}}[n]$ for $t<0$.
\end{example}

\begin{example}\label{exam:counter}
We shall give an example where the quantization 
 $K\in \Derlb(\cor_{M\times M\times I})$ of a Hamiltonian isotopy does not belong to
$\Derb(\cor_{M\times M\times I})$.
Let us take the $n$-dimensional unit sphere $M=\mathbb{S}^n$ ($n\ge2$)
endowed  with the canonical Riemannian metric.
The metric defines an isomorphism
$T^*M\simeq TM$ and the length function $f\cl \dT^*M\to\R$.
Then $f$ is a positive-valued function on $\dT^*M$
homogeneous of degree $1$.
Set $I=\R$ and let $\Phi=\{\phi_t\}_{t\in I}$ be the Hamiltonian
isotopy associated with $f$. 
Then for $p\in \dT^*M$, $\{\pi_M(\phi_t(p))\}_{t\in I}$
is a geodesic. For $x$, $y\in M$, $\dist(x,y)$ 
denotes the distance between $x$ and $y$.
Let $a\cl M\to M$ be the antipodal map.
Then we have $\dist(x,y)+\dist(x,y^a)=\pi$.
For any integer $\ell$ we  set
\eqn
C_{\ell}=\left\{
\ba{ll}
\set{(x,y,t)\in M\times M\times \R}%
{\text{$t\ge\ell\pi$ and $\dist(x,a^\ell(y))\le t-\ell\pi$}}
&\text{if $\ell\ge0$,}\\[2ex]
\Bigl\{(x,y,t)\in M\times M\times\R\ ;\
\parbox{31.5ex}%
{$t<(\ell+1)\pi$ and \hfill\\$\dist(x,a^{\ell+1}(y))<-t+(\ell+1)\pi$}\Bigr\}
&\text{if $\ell<0$.}
\ea
\right.
\eneqn
Let $K$ be the quantization of $\Phi$. Then we have
\eqn
H^k(K)&\simeq&
\begin{cases}
\cor_{C_\ell}&\text{if $k=(n-1)\ell$ for some $\ell\in\Z_{\ge0}$,}\\
\cor_{C_\ell}&\text{if $k=(n-1)\ell-1$ for some $\ell\in\Z_{<0}$,}\\
0&\text{otherwise.}
\end{cases}
\eneqn
\end{example}

\subsection{Deformation of the microsupport}
We consider $\Phi\cl \dTM\times I\to \dTM$ satisfying
hypothesis~\eqref{hyp:isot1}
as in Theorem~\ref{th:3}, we denote as usual by $\Lambda$ the
Lagrangian submanifold associated to its graph and define $\Lambda_t$
as in \eqref{eq:lambdat}.
Let $S_0\subset \dTM$ be a closed conic subset. 
We set $S=\Lambda\circ S_0$ and $S_t=\Lambda_t\circ S_0$,
closed conic subsets of  $\dTM\times T^*I$ and  $\dTM$ respectively. 
We let $i_t\cl M\to M\times I$ be
the natural inclusion for $t\in I$. Consider the functor
\eq\label{eq:pullback-at-t}
\opb{i_t} \cl \Derlb_{S\cup T^*_{M\times I}(M\times I)}(\cor_{M\times I})
\to \Derlb_{S_t \cup T^*_MM}(\cor_M).
\eneq
\begin{proposition}\label{pro:famLagr0}
For any $t\in I$ the functor~\eqref{eq:pullback-at-t} is an equivalence
of categories. In particular for any $F\in \Derlb(\cor_M)$ such that
$\SSi(F) \subset S_t\cup T^*_MM$ there exists a unique
$G\in \Derlb(\cor_{M\times I})$ such that $\opb{i_t}(G) \simeq F$ and
$\SSi(G) \subset S\cup T^*_{M\times I}(M\times I)$.
If $\Phi$ also satisfies~\eqref{hyp:isot2} and we replace $I$ by a relatively
compact subinterval, then~\eqref{eq:pullback-at-t} induces an equivalence
between bounded derived categories.
\end{proposition}
\begin{proof}
Replacing $\Phi$ by $\Phi\circ\opb{\phi_t}$ we may as well assume that
 $t=0$.
Let $K\in \Derlb(\cor_{M\times M\times I})$ be the quantization of 
$\Phi$.
Then we obtain the commutative diagram:
$$
\xymatrix@C=2cm{
\Derlb_{(S_0\times T^*_II)\cup T^*_{M\times I}(M\times I)}(\cor_{M\times I})
\ar[r]^{K \conv\vert_I \cdot}_\sim  \ar[d]_{r_0}  
& \Derlb_{S\cup T^*_{M\times I}(M\times I)}(\cor_{M\times I})
 \ar[d]^{\opb{i_0}}  \\
\Derlb_{S_0 \cup T^*_MM}(\cor_M)   \ar@{=}[r] &
\Derlb_{S_0 \cup T^*_MM}(\cor_M)  ,}
$$
where $r_0$ is induced by $\opb{i_0}$.
It is known that $r_0$ is an equivalence of categories, with inverse given by
$\opb{q_1}$, where $q_1\cl M\times I \to M$ is the projection.  Hence so
is the morphism
$\opb{i_0}$ defined by~\eqref{eq:pullback-at-t}.

\medskip\noindent 
The last assertion follows from Remark~\ref{rem:compact-case-bounded}.
\end{proof}

Now we consider a closed conic Lagrangian submanifold $S_0 \subset \dTM$
and a deformation of $S_0$ indexed by $I$, $\Psi\cl S_0\times I \to \dTM$ as
in Definition~\ref{def:family-Lagr}.  We let $S\subset \dTM\times T^*I$ be the
corresponding Lagrangian submanifold defined in Lemma~\ref{lem:family-Lagr}.
Then Propositions~\ref{prop:extending-isotopies} and~\ref{pro:famLagr0}
imply the following result.
\begin{corollary}\label{thm:famLagr}
We consider a deformation $\Psi\cl S_0\times I \to \dTM$ as in 
Definition~\ref{def:family-Lagr} and we
assume that it satisfies~\eqref{hyp:famLagr}.
Then for any $t\in I$ the functor~\eqref{eq:pullback-at-t} is an equivalence
of categories. 
Moreover if we replace $I$ by a relatively compact subinterval,
it induces an equivalence between the bounded derived categories.
\end{corollary}

\section{Applications to non-displaceability}
\label{sec:applications}
We denote by $\Phi=\{\phi_t\}_{t\in I}\cl \dTM\times I\to \dTM$ 
a homogeneous Hamiltonian isotopy as in Theorem~\ref{th:3}.
Hence, $\Phi$ satisfies hypothesis~\eqref{hyp:isot1}.

Let $F_0\in\Derb(\cor_M)$. We assume
\eq\label{hyp:morse1}
&&\text{$F_0$ has a compact support.}
\eneq
We let $\Lambda \subset \dT^*(M\times M\times I)$ be the conic Lagrangian
submanifold associated to $\Phi$ in Lemma~\ref{lem:homog-Hamilt-isot} and
we let $K\in\Derb(\cor_{M\times M\times I})$ be the quantization of
$\Phi$ on $I$ constructed in Theorem~\ref{th:3}. We set:
\eq\label{eq:phiKF}\ba{rcl}
&&F=K\conv F_0\in\Derb(\cor_{M\times I}),\\
&&F_{t_0}=F\vert_{\{t=t_0\}} \simeq K_{t_0} \conv F_0  \in\Derb(\cor_{M})
\quad\text{for $t_0\in I$.}
\ea\eneq
Then 
\eq\label{eq:propertiesFt}
&&\left\{
\parbox{65ex}{
$\SSi(F) \subset (\Lambda\circ \SSi(F_0)) \cup T^*_{M\times I}(M\times I)$,
 \\[1ex]
$\SSi(F)\cap T^*_MM\times T^*I\subset T^*_{M\times I}(M\times I)$,
\\[1ex]
 the projection $\Supp (F) \to I$ is proper.
}\right.
\eneq
The first assertion follows from \eqref{eq: estss},
and the second assertion follows from the first.
The third one follows from
Proposition~\ref{prop:support_unicity}~(i) and~\eqref{hyp:morse1}.
In particular we have
\eq\label{eq:propertiesFt2}
&&\left\{
\parbox{60ex}{
$F_t$ has a compact support in $M$,\\[1ex]
$\SSi(F_t)\cap\dTM=\phi_t(\SSi(F_0)\cap\dTM)$,
}\right.
\eneq
where the last equality follows from~\eqref{eq:propertiesFt} applied
to $\Phi$ and $\{\opb{\phi_t}\}_{t\in I}$.

\subsection{Non-displaceability: homogeneous case}
We consider a $\Cd^1$-map $\psi\cl M \to \R$
and we assume that 
\eq\label{hyp:morse2}
&&\mbox{the differential $d\psi(x)$ never vanishes.}
\eneq
Hence  the section of $T^*M$ defined by $d\psi$
\eq
&& \Lambda_\psi \eqdot \{(x;d\psi(x)); \; x\in M\}
\label{eq:lambdapsi}
\eneq
is contained in $\dTM$.
\begin{theorem}\label{th:ndispl1}
We consider $\Phi=\{\phi_t\}_{t\in I}$ satisfying~\eqref{hyp:isot1}, 
$\psi\cl M\to\R$ satisfying~\eqref{hyp:morse2} and
$F_0\in\Derb(\cor_M)$ with compact support. We  assume $\rsect(M;F_0)\not=0$.
Then for any $t\in I$,
$\phi_t(\SSi(F_0)\cap\dTM) \cap \Lambda_\psi \neq \emptyset$.
\end{theorem}
\begin{proof}
Let $F,F_t$ be as in~\eqref{eq:phiKF}. Then $F_t$ has compact support and $\rsect(M;F_t)\not=0$
by Corollary~\ref{cor:rsectFt}.
Since 
$\SSi(F_t) \subset \phi_t(\SSi(F_0)\cap\dTM) \cup T^*_MM$, the result follows from Corollary~\ref{cor:Morse1}.
\end{proof}

\begin{corollary}\label{cor:ndispl1}
Let $\Phi=\{\phi_t\}_{t\in I}$ and $\psi\cl M\to\R$ 
be as in Theorem~\ref{th:ndispl1}.
Let $N$ be a non-empty compact  submanifold of $M$. 
Then for any $t\in I$,
$\phi_t(\dot T^*_NM) \cap \Lambda_\psi \neq \emptyset$.
\end{corollary}

\subsection{Non-displaceability: Morse inequalities}
In this subsection and in subsection~\ref{sec:sympl-case} below we assume that
$\cor$ is a field.
Let $F_0\in\Derb(\cor_M)$ and set
\eqn
&&S_0=\SSi(F_0)\cap\dTM.
\eneqn
Now we consider the hypotheses
\eq\label{hyp:morse5a}
&&\parbox{63ex}{$\psi$ is of class $\Cd^2$ and 
the differential $d\psi(x)$ never vanishes,}\\
\label{hyp:morse5b}
&&\left\{
\parbox{63ex}{there exists an open 
subset $S_{0,\reg}$ of $S_0$
such that $S_{0,\reg}$ is a Lagrangian submanifold of class $\Cd^1$
and $F_0$ is a simple sheaf along $S_{0,\reg}$.
}\right.
\eneq
\begin{lemma}\label{le:puretransv}
Let $\Lambda$ be a smooth Lagrangian submanifold defined in a
neighborhood of $p\in\dTM$, let $G\in\Derb(\cor_M)$ and assume $G$ is
simple along $\Lambda$ at $p$. Assume~\eqref{hyp:morse5a} 
and assume that $\Lambda$ and $\Lambda_\psi$
intersect transversally at $p$. Set $x_0=\pi(p)$. Then 
\eqn
&&\sum_j\dim H^j(\rsect_{\{\psi(x)\geq\psi(x_0)\}}(G))_{x_0}=1.
\eneqn
\end{lemma}
\begin{proof}
By the definition
(\cite[Definition~7.5.4]{KS90}),
$\rsect_{\{\psi(x)\geq\psi(x_0)\}}(G)_{x_0}$ is
concentrated in a single degree and its cohomology in this degree has
rank one.
\end{proof}
 In the sequel, for a finite set $A$, we denote by $\# A$ its cardinal.
\begin{theorem}\label{th:morse1}
We consider $\Phi=\{\phi_t\}_{t\in I}$ satisfying \eqref{hyp:isot1}, $\psi\cl M\to\R$ 
satisfying~\eqref{hyp:morse5a} and
$F_0\in\Derb(\cor_M)$ with compact support. We also assume~\eqref{hyp:morse5b}.
Let  $t_0\in I$. Assume that 
$\Lambda_\psi\cap\phi_{t_0}(S_0)$ is contained in 
$\Lambda_\psi\cap\phi_{t_0}(S_{0,\reg})$ and the intersection is
finite and transversal. Then 
\eqn
&&\#\bl\phi_{t_0}(S_0) \cap \Lambda_\psi\br\ge\sum_jb_j(F_0).
\eneqn
\end{theorem}
\begin{proof}
It follows from Corollary~\ref{cor:rsectFt} that $b_j(F_t)=b_j(F_0)$ for
all $j\in\Z$ and all $t\in I$. 

Let $\{q_1,\dots,q_L\}=\Lambda_\psi\cap\phi_{t_0}(\Lambda_0)$,
$y_i=\pi(q_i)$ and
set 
\eqn
&&W_i\eqdot\rsect_{\{\psi(x)\geq\psi(y_i)\}}(F_{t_0})_{y_i}.
\eneqn
By Lemma~\ref{le:puretransv}, $W_i$ is a bounded
complex with finite-dimensional cohomologies 
 and it follows from
the Morse inequalities~\eqref{eq:morseineq2} that
\eqn
&&\sum_jb_j(F_{t_0})\leq \sum_j\sum_ib_j(W_i).
\eneqn
Moreover
\eqn
&&\sum_j\dim H^j(\rsect_{\{\psi(x)\geq\psi(y_i)\}}(F_{t_0})_{y_i})=1
\quad\text{for any $i$,}
\eneqn
and it implies
\eqn
&&\sum_j\sum_i b_j(W_i)=\#\bl\SSi(F_{t_0})\cap \Lambda_\psi\br
=\#\bl\phi_{t_0}(\SSi(F_0)\cap\dTM) \cap \Lambda_\psi\br.
\eneqn
\end{proof}

\begin{corollary}\label{cor:ndispl2}
Let $\Phi=\{\phi_t\}_{t\in I}$, let $\psi\cl M\to\R$ 
satisfying~\eqref{hyp:morse5a} and let $N$ be a compact submanifold of $M$.
Let  $t_0\in I$. Assume that 
$\phi_{t_0}(\dT_N^*M)$ and $\Lambda_\psi$ intersect transversally. Then
\eqn
&&\#(\phi_{t_0}(\dT_N^*M) \cap \Lambda_\psi)\ge
\sum_j \dim H^j(N;\cor_N).
\eneqn
\end{corollary}
\begin{proof}
Apply Theorem~\ref{th:morse1} with $F_0=\cor_N$. 
\end{proof}
\begin{remark}\label{rem:corner}
Corollaries~\ref{cor:ndispl1} and~\ref{cor:ndispl2}
extend to the case where $N$ is replaced
with a compact submanifold with boundary or even with corners. In this
case, one has to replace the conormal bundle $T^*_NM$ with the
microsupport of the constant sheaf $\cor_N$ on $M$. Note that this microsupport
is easily calculated. 
For Morse inequalities on manifolds with
boundaries, see the recent paper~\cite{L10} and 
see also~\cite{KO01,O98} for related results.
\end{remark}

\subsection{Non-displaceability:  non-negative Hamiltonian isotopies}
Consider as above a manifold $M$ and $\Phi\cl \dTM\times I\to \dTM$ a
homogeneous Hamiltonian isotopy, 
that is, $\Phi$ satisfies~\eqref{hyp:isot1}.
We define $f\cl \dTM\times I\to\R$ homogeneous of degree $1$
and $\Lambda\subset \dTM\times\dTM\times T^*I$ as
in Lemma~\ref{lem:homog-Hamilt-isot}.
The following definition is due to~\cite{EKP06}
and is used in~\cite{CN10, CFP10} where the authors prove
Corollary~\ref{cor:positiso} below 
in the particular case where $X$
and $Y$ are points and other related results.

\begin{definition}\label{def:positisot}
The isotopy $\Phi$ is said to be non-negative if 
$\langle \alpha_M, H_f \rangle \geq 0$.
\end{definition}
Let $\eu_M$ be the Euler vector field on $T^*M$. 
Then $\langle\alpha_M,H_f\rangle=\eu_M(f)$ and since $f$ is 
of degree $1$ we have $\eu_M(f)=f$. Hence
$\Phi$ is non-negative if and only if $f$ is a non-negative valued function.  
We let $(t,\tau)$ be the
coordinates on $T^*I$. Then by~\eqref{eq:def-lambda} this condition is  also
equivalent to 
\eqn
&&\Lambda\subset \{\tau \leq 0 \}. 
\eneqn

In order to prove Theorem~\ref{th:positiso} below,
we shall give several results in sheaf theory.

\begin{proposition}\label{pro:microsupp}
Let $N$ be a manifold, $I$ an open interval of $\R$ containing $0$. Let
$F\in\Derb(\cor_{N\times I})$ and, for $t\in I$, 
set $F_t=F|_{N\times\{t\}}\in\Derb(\cor_N)$.
Assume that
\banum
\item
$\SSi(F)\subset\{\tau\leq 0\}$,
\item
$\SSi(F)\cap (T^*_NN\times T^*I)\subset
 T^*_{N\times I}(N\times I)$,
\item
$\Supp(F)\to I$ is proper.
\eanum
Then we have:
\bnum
\item
for all $a\leq b$ in $I$ there are natural morphisms
$r_{b,a}\cl F_a\to F_b$,
\item $r_{b,a}$ induces a commutative diagram of isomorphisms 
$$\xymatrix{
{\rsect(N\times I;F)}\ar[d]^\wr\ar[dr]^\sim\\
{\rsect(N;F_a)}\ar[r]^{\sim}_{r_{b,a}}& {\rsect(N;F_b).}
}$$
\enum
\end{proposition}
\begin{proof}
By a similar argument to Corollary~\ref{cor:rsectFt}, 
(b) and (c) imply that
$\rsect(N\times I;F)\to\rsect(N;F_t)$ is an isomorphism
for any $t\in I$.

\medskip\noindent
For $b\in I$ set $I_b=\set{t\in I}{t\le b}$
and $F'=F\otimes\cor_{N\times I_b}$.
Then $F'$ also satisfies (a).
Hence \cite[Prop. 5.2.3]{KS90} implies that
$F'\simeq F'\conv\cor_D$,
where
\eqn
&&\DD=\set{(s,t)\in I\times I}{t\le s}.
\eneqn
We deduce the isomorphisms, for any $a\in I_b$\;:
\eq\label{eq:microsupp2}
&&F_{a} \simeq F' \conv \cor_{\{a\}} 
\simeq F'\conv\cor_{\DD}\conv\cor_{\{a\}} 
\simeq F'\conv\cor_{[a,b]}
\simeq F\conv\cor_{[a,b]}.
\eneq
The morphism $r_{b,a}$ is then induced by the morphism 
$\cor_{[a,b]} \to \cor_{\{b\}}$.
Hence we obtain 
a commutative diagram 
\eqn
&&\xymatrix{
&\rsect(N\times I;F)\ar[dl]_\sim\ar[dr]^\sim\\
\rsect(N\times [a,b];F)\ar[d]^\wr\ar[rr]
&&\rsect(N\times \{b\}\;;F)\ar[d]^\wr\\
\rsect(N;F\circ \cor_{[a,b]})\ar[rr]\ar[d]^\wr&&\rsect(N;F\circ \cor_{\{b\}})
\ar[d]^\wr\\
\rsect(N;F_a)\ar[rr]^{r_{b,a}}&&\rsect(N;F_b).
}
\eneqn
\end{proof}

We recall that $\omega_X$ denotes the dualizing complex of a manifold
$X$.
\begin{lemma}\label{le:positiso}
Let $M$ be a manifold and $X$ a locally closed subset of $M$.
Let $i_X\cl X\to M$ be the embedding.
We assume that the base ring $\cor$ is not reduced to $\{0\}$.
\bnum
\item
Let $F\in\Der(\cor_M)$ and assume that there exists a morphism
$u\cl F\to\roim{i_X}\cor_X$ which induces an isomorphism
$H^0(M;F)\isoto H^0(M;\roim{i_X}\cor_X)$. Then $X\subset\Supp(F)$.
\item
Let $G\in\Der(\cor_M)$ and  assume that there exists a morphism 
$v\cl \eim{i_X}\omega_X\to G$ which induces an isomorphism
$H^0_c(M;\eim{i_X}\omega_X)\isoto H^0_c(M;G)$. 
Then $X\subset\Supp(G)$.
\enum
\end{lemma}
\begin{proof}
(i) Let $x\in X$ and let $i_x\cl\{x\}\hookrightarrow M$
be the inclusion. For $x\in X$, the composition $\cor\to H^0(M;\roim{i_X}\cor_X)\to \cor$
is the identity. Hence, in the commutative diagram
\eqn
&&\xymatrix{
H^0(M;F)\ar[r]^-{\sim}_-{u}\ar[d]_{i_x^{-1}}
                  &H^0(M;\roim{i_X}\cor_X)\ar@{->>}[d]_{i_x^{-1}}\\
H^0(F)_x\ar[r]    &{\cor}
}\eneqn
the map $H^0(F)_x\to \cor$ is surjective. We conclude that 
$x\in\Supp(F)$.

\medskip\noindent
(ii) For $x\in X$, the composition 
$H^0_{\{x\}}(M; \eim{i_X}\omega_X) \to  H^0_c(M; \eim{i_X}\omega_X) \to \cor$
is an isomorphism.
Hence in the commutative diagram induced by $v$
\eqn
&&\xymatrix{
H^0_{\{x\}}(M; \eim{i_X}\omega_X) \ar[r]\ar@{>->}[d]_a& H^0_{\{x\}}(M;G)\ar[d] \\
H^0_c(M; \eim{i_X}\omega_X) \ar[r]_-{b}^-{\sim}& H^0_c(M;G)\;,
}
\eneqn
the morphism $a$ is injective and $b$ is bijective.
Hence $\cor\simeq H^0_{\{x\}}(M; \eim{i_X}\omega_X)\to H^0_{\{x\}}(M;G)$
is injective.
Therefore $H^0_{\{x\}}(M;G)$ does not vanish and 
$x\in\Supp(G)$.
\end{proof}

\begin{lemma}\label{lem:geometry-positiso}
Let $M$ be a non-compact connected manifold and let $X$ be a compact
connected submanifold of $M$. Then we have:
\bnum
\item
The open subset $M\setminus X$ has at most two connected
components. 
\item
Assume that 
there exists a relatively compact connected component $U$
of $M\setminus X$.
Then such a connected component is unique,
$X$ is a hypersurface and $X$ coincides with the boundary of 
$U$.
\enum
\end{lemma}
\Proof
(i)\quad We have an exact sequence 
$$H^0(M;\C)\to H^0(M\setminus X;\C)\to H^1_X(M;\C).$$
The last term $H^1_X(M;\C)$ is isomorphic to $H^0(X;\lh^1_X(\C_M))$.
Since $\lh^1_X(\C_M)$ is locally isomorphic to $\C_X$ or $0$,
we have $\dim H^1_X(M;\C)\le 1$. Hence we obtain
$\dim H^0(M\setminus X;\C)\le 2$.
Hence $M\setminus X$ has at most two connected
components.

\noindent
(ii)\quad 
Assume that there exists a relatively compact connected component $U$
of $M\setminus X$. If $M\setminus X$ has another
relatively compact connected component $V$, then $M=X\cup U\cup V$ 
by (i) and it is compact. It is a contradiction.
Hence  a relatively compact connected component $U$
of $M\setminus X$ is unique if it exists.
If $X$ is not a hypersurface then $M\setminus X$ is connected and not
relatively compact. It is a contradiction . Hence $X$ is a hypersurface.
Then it is obvious that $X$ coincides with the boundary of $U$.
\QED

Until the end of this subsection, we assume
that $\Phi=\{\phi_t\}_{t\in I}\cl\dT^*M\times I\to\dT^*M$ is
a non-negative homogeneous Hamiltonian isotopy.

We define $g\cl\dTM\times I\to \R$
by
\eq\label{eq:inv-opp-isot}
g(p,t) = f(\phi_t(p)^a, t)\quad (p\in\dT^*M,\ t\in I).
\eneq
Here $a\cl \dTM\to \dTM$ is the antipodal map.

\begin{lemma}\label{lem:inv-opp-isot}
Let $\Psi$ be the symplectic isotopy given by
$\Psi=\{a\circ \opb{\phi_t} \circ a\}_{t\in I}$.
Then we have $\dfrac{\partial\Psi}{\partial t}=H_{g_t}$,
and $\Psi$ is a non-negative Hamiltonian isotopy.
\end{lemma}
\begin{proof}
Set $\psi_t=a\circ \opb{\phi_t}\circ a$.
Let $\Lambda$ be the Lagrangian manifold associated to $\Phi$ as in 
Lemma \ref{lem:Hamilt-isot}:
\eqn
\Lambda & = &
\set{\bl\phi_t(v), v^a, t, -f(\phi_t(v),t)\br}{ v\in\dT^*M, t\in I}.
\eneqn
Then we have
\eqn
\Lambda & = &
\set{\bl w,  \phi_t^{-1}(w)^a, t, -f(w,t)\br}{ w\in\dT^*M, t\in I}\\
&=&
\set{\bl w^a,  \psi_t(w), t, -f(w^a,t)\br}{ w\in\dT^*M, t\in I}.
\eneqn
Since $f(w^a,t)=g(\phi_t^{-1}(w^a)^a,t)$, the set
$$\set{\bl w^a, \psi_t(w), t, -g(\psi_t(w),t)\br}{ w\in\dT^*M, t\in I}$$
is Lagrangian.
Hence Lemma~\ref{lem:Hamilt-isot} implies that
$\partial\Psi/\partial t=g_t$.
The non-negativity of $\Psi$ is obvious since $g$ itself is non-negative. 
\end{proof}

\begin{lemma}\label{lem:stable-trivial}
Let $\Lambda_1$ and $\Lambda_2$ be conic closed Lagrangian submanifolds of
$\dT^*M$. If either $\phi_t(\Lambda_1)\subset \Lambda_2$ for all
$t\in[0,1]$, or if $\Lambda_1\subset \phi_t(\Lambda_2)$ for all
$t\in[0,1]$, then $\phi_t|_{\Lambda_1} = \id_{\Lambda_1}$ for all $t\in[0,1]$.
\end{lemma}
\begin{proof}
(i) Let us treat the case where $\phi_t(\Lambda_1)\subset \Lambda_2$ for all $t\in[0,1]$.
 We may assume that $\Lambda_1$ is connected.
Then
$\phi_t(\Lambda_1)$ is a connected component of $\Lambda_2$, hence
does not depend on $t$.
Therefore $\phi_t(\Lambda_1)=\Lambda_1$ for all $t\in[0,1]$.
The hypothesis implies that $H_{f_t} = \partial\Phi/\partial t$ is tangent to
$\Lambda_1$ for all $t\in[0,1]$. By Lemma~\ref{lem:homog-Hamilt-isot},
$f_t = \langle \alpha_M, H_{f_t} \rangle$. 
Since $\Lambda_1$ is conic Lagrangian, the Liouville form $\alpha_M$ vanishes
on the tangent bundle of $\Lambda_1$ and we deduce that $f$ is identically $0$
on $\Lambda_1\times [0,1]$.

Since  $f_t$ is a non-negative function on $\dTM$, all points of 
$\Lambda_1$ are minima of $f$. It follows that $d(f_t) = 0$ on $\Lambda_1$ for all $t\in[0,1]$.
Hence $H_{f_t}$ also vanishes on $\Lambda_1$ and therefore
$\phi_t|_{\Lambda_1} = \id_{\Lambda_1}$ for all $t\in[0,1]$.

\noindent
(ii) Now assume that $\Lambda_1\subset\phi_t(\Lambda_2)$ for all $t\in[0,1]$. 
Set $\psi_t=a\circ\phi_t^{-1}\circ a$.
Then $\{\psi_t\}_{t\in I}$ is a non-negative Hamiltonian isotopy
by Lemma~\ref{lem:inv-opp-isot},
and $\psi_t(\Lambda_1^a)\subset\Lambda_2^a$ holds for any $t\in [0,1]$.
Hence step~(i) implies that $\psi_t\vert_{\Lambda_1^a}=\id_{\Lambda_1^a}$.
\end{proof}

\begin{theorem}\label{th:positiso}
Let $M$ be a connected and non-compact manifold and let $X,Y$ be two compact
connected submanifolds of $M$.  
Let $\Phi=\{\phi_t\}_{t\in I}\cl\dT^*M\times I\to\dT^*M$ be
a non-negative homogeneous Hamiltonian isotopy.
Assume that $[0,1]\subset I$ and $\phi_1(\dT^*_XM)=\dT^*_YM$. Then $X=Y$ and 
$\phi_t|_{\dT^*_XM} = \id_{\dT^*_XM}$ for all $t\in[0,1]$.\footnote{
In an earlier draft of this paper, we only proved the first part of the
conclusion of Theorem~\ref{th:positiso}, namely that $X=Y$. We 
thank Stephan Nemirovski who asked us the question whether 
$\phi_t\vert_{\dT^*_XM}$ is the identity of $\dT^*_XM$ for all $t\in[0,1]$.
}
\end{theorem}

\begin{proof}
By Lemma~\ref{lem:stable-trivial} it is enough to prove that
$X=Y$ and $\phi_t(\dT^*_XM)\subset \dT^*_XM$ for all $t\in[0,1]$.

We will distinguish two cases (see Lemma~\ref{lem:geometry-positiso}), respectively treated in~(ii)
and~(iii) below:
\banum
\item $M\setminus X$ or $M\setminus Y$ has no relatively
compact connected component,
\item both $X$ and $Y$ are the boundaries of relatively compact connected open
  subsets $U$ and $V$ of $M$, respectively.
\eanum

\medskip
\noindent
(i) Let $K\in\Derlb(\cor_{M\times M\times I})$ be the quantization of $\Phi$
on $I$ given by Theorem~\ref{th:3}. 
By Proposition~\ref{prop:support_unicity}~(ii),
the convolution with $K_1$ gives an equivalence of categories
\eqn
&&\Derlb_{T^*_XM\cup T^*_MM}(\cor_M)
\isoto[\;K_1\conv\cdot\ ]\Derlb_{T^*_YM \cup T^*_MM}(\cor_M).
\eneqn
Moreover
$\SSi(K)\subset\Lambda\cup T^*_{M\times M\times I}(M\times M\times I)$,
so that $\SSi(K)\subset \{ \tau \leq 0 \}$. 
We consider $F_0\in \Derb(\cor_M)$ with compact support.
We set:
\eqn
&&F=K\conv F_0,\quad F_{t_0}=F\conv\cor_{\{t=t_0\}} \, (t_0\in I).
\eneqn
Then $F$ satisfies~\eqref{eq:propertiesFt} and we have
$\SSi(F)\subset \{ \tau \leq 0 \}$.  Hence we may apply
Proposition~\ref{pro:microsupp} and we deduce that,
for all $a,b\in I$ with $a\leq b$,
there are natural morphisms $r_{b,a}\cl F_a\to F_b$ which induce isomorphisms
$\rsect(M;F_a) \isoto \rsect(M;F_b)$.

\medskip\noindent
(ii) Let us assume hypothesis~(a).
By Lemma~\ref{lem:inv-opp-isot}, 
replacing $\phi_t$ with $a\circ\opb{\phi_t}\circ a$
and $X$ with $Y$ if necessary,
we may assume that any of the connected
components of $M\setminus X$ is not relatively compact.

\smallskip\noindent
(ii-a) Let us show that $X=Y$.
There exists $F_0\in \Derlb(\cor_M)$ such that $F_1\simeq \cor_Y$.
We have $F_0 \simeq K_1^{-1} \conv \cor_Y$ so that $F_0$ has compact support.
We have also $\SSi(F_t)\cap \dT^*M=\phi_t(\dT^*_XM)$.
Since $\SSi(F_0) \subset T^*_XM \cup T^*_MM$, $F_0$ is locally constant
outside $X$. Since $M\setminus X$ has no compact connected component, we
deduce $\Supp(F_0) \subset X$.
Hence by Lemma~\ref{le:positiso} (i),
we have $Y\subset \Supp(F_0)\subset X$.

Since $M\setminus X\subset M\setminus Y$ and $M\setminus X$ has no relatively compact connected component,
$M\setminus Y$ has also no relatively compact connected component.
Hence by interchanging $X$ and $Y$ with the use of Lemma~\ref{lem:inv-opp-isot}, 
we obtain $X\subset Y$.
Thus we obtain $X=Y$.

\medskip\noindent
(ii-b) Let us show
$\dT^*_XM\subset\phi_t(\dT^*_XM)$.
Assuming that $p\in \dT^*_XM \setminus \phi_t(\dT^*_XM)$,
let us derive a contradiction.
Take a $\Cd^1$-function $g$ such that
$p=(x;dg(x))$ and $g\vert_X=0$. 
Since $\Supp(F_0)\cap \{g<0\}=\emptyset$, we obtain $\lh^0_{\{g<0\}}(F_0)_x\simeq0$.
Since $dg(x)\not\in\SSi(F_t)$, the morphism
$\lh^0(F_t)_x\to\lh^0_{\{g<0\}}(F_t)_x$ is an isomorphism.
Then we have a commutative diagram
\eqn
&&\xymatrix{
&H^0(M; F_0) \ar[r]^\sim \ar[d] & 
H^0(M; F_t) \ar[r]^\sim \ar[d] & 
H^0(M; \cor_Y) \ar[d]^{\bwr}\ar[r]^-{\sim}&\cor  \\
&H^0(F_0)_x\ar[d] \ar[r]&\lh^0(F_t)_x\ar[d]^-{\bwr} \ar[r] & (\cor_Y)_x \\
0\ar@{-}[r]^-{\sim}&\lh^0_{\{g<0\}}(F_0)_x\ar[r]&
\lh^0_{\{g<0\}}(F_t)_x}
\eneqn
Hence $\cor\simeq H^0(M; F_t)\to\lh^0_{\{g<0\}}(F_t)_x$ is a monomorphism and 
also the zero morphism. This is a contradiction.
Thus we obtain the desired result $\dT^*_XM\subset \phi_t(\dT^*_XM)$.
Thanks to Lemma~\ref{lem:stable-trivial}, this completes the proof of the theorem under hypothesis~(a).

\bigskip\noindent
(iii) Now we assume hypothesis~(b).
In this case, $X$ and $Y$ are hypersurfaces of $M$.
Let $\dTi_YM$ be the ``inner'' conormal of $Y$, so that
$\SSi(\cor_{\ol V}) = \ol{V} \cup \dTi_YM$ (see
Example~\ref{ex:microsupp}).

\smallskip\noindent
(iii-a) Let us first prove that $U=V$. As in~(ii-a) there exists
$F_0\in \Derlb(\cor_M)$ with compact support such that
$F_1\simeq \cor_{\ol V}$. As in~(ii-a) we see that $\Supp(F_0) \subset 
\ol{U}=U\cup X$.
Part~(i) gives a morphism $r_{1,0}\cl F_0\to \cor_{\ol V}$ which induces
$H^0(M; F_0) \isoto H^0(M; \cor_{\ol V})\isoto\cor$.  
Hence Lemma~\ref{le:positiso} implies that
$\ol{V}\subset \Supp(F_0) \subset \ol{U}$.
Then Lemma~\ref{lem:inv-opp-isot} implies the reverse inclusion.
Hence $U = V$ and $X=Y$.

\medskip\noindent
(iii-b)  Let us prove that $\dTi_XM\subset\phi_t\phi_1^{-1}(\dTi_XM)$ 
for all $t\in [0,1]$.
The proof is similar to the one in (ii-b).
Assuming there exist $t\in[0,1]$ and 
$p\in(\dTi_XM)\setminus \phi_t\phi_1^{-1}(\dTi_XM)$,
we shall derive a contradiction.
Write $p=(x;dg(x))$ for a $\Cd^1$-function $g$ such that
$g\vert_X=0$. Hence $\{g>0\}$ coincides with $U$ on a neighborhood of $x$.
Then we have a commutative diagram
\eqn
&&\xymatrix{
&H^0(M; F_0) \ar[r]^\sim \ar[d] & 
H^0(M; F_t) \ar[r]^\sim \ar[d] & 
H^0(M; \cor_{\ol{V}}) \ar[d]^{\bwr}\ar[r]^-{\sim}&\cor  \\
&H^0(F_0)_x\ar[d] \ar[r]&H^0(F_t)_x\ar[d]^-{\bwr} \ar[r] & (\cor_{\ol{V}})_x \\
0\ar@{-}[r]^-{\sim}&\lh^0_{\{g<0\}}(F_0)_x\ar[r]&
\lh^0_{\{g<0\}}(F_t)_x}
\eneqn
Hence $\cor\simeq H^0(M; F_t)\to\lh^0_{\{g<0\}}(F_t)_x$ is a monomorphism and
equal to the zero morphism.
This is a contradiction.
Hence we obtain $\dTi_XM\subset\phi_t\phi_1^{-1}(\dTi_XM)$,
or equivalently $\phi_t^{-1}(\dTi_XM) \subset\phi_1^{-1}(\dTi_XM)$.
Hence Lemma~\ref{lem:stable-trivial}
implies that $\phi_t^{-1}\vert_{\dTi_XM}=\id_{\dTi_XM}$, or
$\phi_t\vert_{\dTi_XM}=\id_{\dTi_XM}$.
Lemma~\ref{lem:inv-opp-isot} permits us
to apply this to $a\circ\phi_t\circ a$,
and we obtain $\phi_t\vert_{a\dTi_XM}=\id_{a\dTi_XM}$.
Thus we obtain $\phi_t\vert_{\dT^*_XM}=\id_{\dT^*_XM}$.
\end{proof}

\begin{corollary}\label{cor:positiso}
Let $M$ be a connected manifold such that
the universal covering $\tM$ of $M$ is
non-compact.
Let $X$ and $Y$ be simply connected and compact submanifolds of $M$
with codimension $\ge2$. 
Let $\Phi\cl\dTM\times I\to \dTM$ be
a non-negative homogeneous Hamiltonian isotopy 
such that $[0,1]\subset I$
and $\phi_1(\dT^*_XM)=\dT^*_YM$. Then $X=Y$ and 
$\phi_t|_{\dT^*_XM} = \id_{\dT^*_XM}$ for all $t\in[0,1]$.
\end{corollary}

\begin{proof}
Let $q\cl\tM\to M$ 
and $p\cl\dT^*\tM\to\dT^*M$ be the canonical projections.
Let $f\cl \dT^*M\times I\to\R$ be as in Lemma~\ref{lem:homog-Hamilt-isot}
for $\Phi$ and set $\tilde{f}\seteq f\circ (p\times\id_{I})
\cl\dT^*\tM\times I\to\R$.
Let $\tilde{\Phi}\cl \dT^*\tM \times I \to \dT^*\tM$ the non-negative
homogeneous Hamiltonian isotopy associated with $\tilde{f}$.
We set 
$\tilde{\phi_t}\seteq\tilde{\Phi}\vert_{\dT^*\tM \times\{t\}}
\cl\dT^*\tM\to\dT^*\tM$. Then $p\circ \tilde{\phi}_t = \phi_t \circ p$.
Let 
$\opb{q}(X) =\bigsqcup_{j\in J}\tilde X_j$ and 
$\opb{q}(Y) =\bigsqcup_{k\in K}\tilde Y_k$
be the decompositions into connected components.
Then by the assumption, $\tilde X_j\to X$
and $\tilde Y_k\to Y$ are isomorphisms, and hence
$\tilde X_j$ and $\tilde Y_k$ are connected and compact.
Since $p^{-1}(\dT^*_XM)=\sqcup_j\dT^*_{\tilde X_j}\tM$, we have
$$\bigsqcup_{j\in J} \tilde{\phi_1}(\dT^*_{\tilde X_j}\tM)
= \bigsqcup_{k\in K} \dT^*_{\tilde Y_k}\tM.$$
Since $\codim \tilde X_j$, $\codim \tilde Y_k>1$,
the unions are decompositions into connected components.
So, for a given $j\in J $, there exists  $k\in K$ such that
$\tilde\phi_1(\dT^*_{\tilde X_j}\tM) = \dT^*_{\tilde Y_k}\tM$. 
Hence Theorem~\ref{th:positiso} implies $\tilde X_j=\tilde Y_k$ and 
$\tilde\phi_t|_{\dT^*_{\tilde X_j}\tM} = \id_{\dT^*_{\tilde X_j}\tM}$
for all $t\in[0,1]$.
Finally we conclude $X=q(\tilde X_j)=q(\tilde Y_k)=Y$ and
$\phi_t|_{\dT^*_XM} = \id_{\dT^*_XM}$ for all $t\in[0,1]$.
\end{proof}

\subsection{Non-displaceability:  symplectic case}
\label{sec:sympl-case}
In this section we assume that $\cor$ is a field.
Using Theorems~\ref{th:ndispl1} and~\ref{th:morse1} we recover a well-known
result solving a conjecture by Arnold~\cite{Cha83, CZ83, Fe97, Ho85, LS85}.
We first state an easy geometric lemma.
\begin{lemma}\label{lem:pullback-pushforward}
Let $p\cl E \to X$ be a smooth morphism, let $A,B$ be submanifolds of $X$
and $A'$ a submanifold of $E$.
We assume that $p$ induces a diffeomorphism $p|_{A'}\cl A'\isoto A$.
We set $B'=\opb{p}(B)$. Then
\eq
&&\text{$p$ induces a bijection $A'\cap B' \isoto A\cap B$}, 
\label{eq:morseNH1}\\
&& \text{\parbox[t]{60ex}{$A'$ and $B'$ intersect transversally 
if and only if $A$ and $B$ intersect transversally.}}
\label{eq:morseNH2}
\eneq
\end{lemma}

\begin{theorem}\label{thm:morseNH}
Let $N$ be a non-empty compact manifold.
Let $\Phi\cl T^*N\times I\to T^*N$ be a
Hamiltonian isotopy and assume that there exists a compact set $C\subset T^*N$ such
that $\Phi|_{(T^*N \setminus C) \times I}$ is the projection on the first
factor. We let $c=\sum_j\dim H^j(N;\cor_N)$, the sum
of the Betti numbers of $N$. 
Then for any $t\in I$ the intersection
$\phi_t(T^*_NN) \cap T^*_NN$ is never empty.
Moreover its cardinality is at
least $c$ whenever the intersection is transversal.
\end{theorem}
\begin{proof}
(i) We set $M=N\times \R$ and identify $N$ with $N\times\{0\}$.
We let $\tw\Phi\cl \dTM \times I \to \dTM$ be the homogeneous
Hamiltonian isotopy given by
Proposition~\ref{prop:homogenisation} and we set
$\tw\phi_t = \tw\Phi(\cdot,t)$.

We apply Theorems~\ref{th:ndispl1} and~\ref{th:morse1} to $M$, $\tw\Phi$,
$F_0 = \cor_N$ and $\psi = t$, the projection from $M$ to $\R$.
We obtain that the intersection $\tw\phi_t(\dT^*_NM) \cap \Lambda_\psi$
is a non-empty set whose cardinality is at least $c$ whenever
the intersection is transversal.
  
\medskip\noindent
(ii) Now we compare $\tw\phi_t(\dT^*_NM) \cap \Lambda_\psi$ with the
intersection considered in the theorem.

\noindent
(ii-a)
We apply Lemma~\ref{lem:pullback-pushforward} with
$X = T^*N$, $E= T^*N \times \R^\times$, $p(x,\xi,\sigma) = (x,\xi/\sigma)$,
$A=T^*_NN$, $B= \phi_t(T^*_NN)$ and $A' = T^*_NN \times \{1\} \subset E$.
We set $\Sigma_t\eqdot B'=\opb{p}(B)$. We have
\eq\label{eq:morseNH3}
&&\Sigma_t=
\{(\sigma\cdot\phi_t(x,0),\sigma)\in T^*N\times\R^\times;\;
x\in N,\sigma\in\R^\times\}.
\eneq
By Lemma~\ref{lem:pullback-pushforward},
$(T^*_NN \times \{1\}) \cap \Sigma_t \isoto T^*_NN  \cap \phi_t(T^*_NN)$
and one of these intersections is transversal if and only if the other one
is.

\smallskip\noindent
(ii-b)
We apply Lemma~\ref{lem:pullback-pushforward} with
$X = T^*N\times \R^\times$, $E = T^*N \times \dT^*\R$,
$p(x,\xi,s,\sigma)=(x,\xi,\sigma)$,
$A=\Sigma_t$, $B = T^*_NN\times\{1\}$ and $A'=\tw\phi_t(\dT^*_NM)$.
We must check that the restriction of $p$ to $\tw\phi_t(\dT^*_NM)$ induces an
isomorphism $\tw\phi_t(\dT^*_NM)\isoto\Sigma_t$.
This follows from~\eqref{eq:morseNH3} and the identity
(see~\eqref{eq:twphi}):
\eqn
&&\tw\phi_t(\dT^*_NM)=
\{(\sigma\cdot\phi_t(x,0),u(x,0,t),\sigma); x\in N,\sigma\in \R^\times\}.
\eneqn
We see easily that $B'=\opb{p}(B)$ is  $\Lambda_\psi$.
Hence Lemma~\ref{lem:pullback-pushforward} implies
$$
\tw\phi_t(\dT^*_NM) \cap \Lambda_\psi \isoto 
\Sigma_t \cap (T^*_NN \times \{1\})
$$
and one of these intersections is transversal if and only if the other one
is. Together with~(ii-a) and~(i) this gives the theorem.
\end{proof}

\appendix

\section{Appendix: Hamiltonian isotopies}
\label{sec:Hisot}

We first recall some notions of symplectic geometry. Let $\symx$ be a
symplectic manifold with symplectic form $\omega$. We denote by $\symx^a$ the
same manifold endowed with the symplectic form $-\omega$.  The symplectic
structure induces the Hamiltonian isomorphism $\h\cl T\symx \isoto T^*\symx$
by $\h(v) = \iota_v(\omega)$, where $\iota_v$ denotes the
contraction with $v$.
To a vector field $v$ on $\symx$ we associate in this way a $1$-form
$\h(v)$ on $\symx$.  For a $\Cd^\infty$-function
$f\cl \symx\to \R$ the Hamiltonian vector field of $f$ is by definition
$H_f \eqdot -\h^{-1}(df)$.

The vector field $v$ is called symplectic if its flow preserves $\omega$.
This is equivalent to $\shl_v(\omega) = 0$ where $\shl_v$ denotes
the Lie derivative of $v$.  By Cartan's formula
($\shl_v= d\,\iota_v+ \iota_v\,d$) this is again equivalent
to $d(\h(v)) = 0$ (recall that $d\omega=0$).  The vector field $v$ is
called Hamiltonian if $\h(v)$ is exact, or equivalently
$v=H_f$ for some function $f$ on $\symx$.

In this section we consider an open interval $I$ of $\R$ containing the
origin. We will use the following general notation: for a map
$u\cl X\times I \to Y$ and $t\in I$ we let $u_t\cl X\to Y$ be the map
$x\mapsto u(x,t)$.

\subsection{Families of symplectic isomorphisms}\label{app:symp}
Let $\Phi\cl \symx \times I\to \symx$ be a $\Cd^\infty$-map such that
$\phi_t\eqdot\Phi(\cdot,t)\cl \symx \to \symx$ is a symplectic isomorphism for
each $t\in I$ and is the identity for $t=0$. The map $\Phi$ induces a time
dependent vector field on $\symx$
\eq
&& v_\Phi \eqdot \frac{\partial\Phi}{\partial t} \cl \symx\times I\to T\symx.
\eneq 
Since $\phi_t^*(\omega) =\omega$ we obtain by derivation
$\shl_{(v_\Phi)_t} (\omega)= 0$ for any $t\in I$, that is, $(v_\Phi)_t$ is a
symplectic vector field.  So the corresponding ``time dependent'' $1$-form
$\beta =\h(v_\Phi) \cl \symx\times I\to T^*\symx$ satisfies $d(\beta_t) =0$
for any $t\in I$.  The map $\Phi$ is called a Hamiltonian isotopy if
$(v_\Phi)_t$ is Hamiltonian, that is, if $\beta_t$ is exact for any $t$. 
In this case, integrating the $1$-form $\beta$ (which is $\Cd^\infty$ with
respect to the parameter $t$) we obtain a $\Cd^\infty$-function
$f\cl \symx \times I\to\R$ such that $\beta_t =- d(f_t)$.
Hence we have
\eq\label{def:f}
&&\frac{\partial\Phi}{\partial t} =H_{f_t}.
\eneq

\medskip

The fact that the isotopy $\Phi$ is Hamiltonian can be interpreted as a
geometric property of its graph as follows.  For a given $t\in I$ we let
$\Lambda_{t}$ be the graph of $\phi_t^{-1}$
and we let $\Lambda'$ be the family of $\Lambda_t$'s:
\eqn
&&\Lambda_t=\set{(\phi_t(v),v)}{v\in\symx^a} \subset \symx\times\symx^a, \\
&& \Lambda'= \set{(\phi_t(v),v,t)}{v\in\symx^a, \; t\in I}
\subset  \symx \times \symx^a\times I.
\eneqn
Then $\Lambda_t$ is a Lagrangian submanifold of $\symx \times \symx^a$
and we ask whether
we can lift $\Lambda'$ as a Lagrangian submanifold $\Lambda$ of
$\symx \times \symx^a \times T^*I$ so that
\eq\label{eq:Lambda-Lambda'}
(\id_{\symx \times \symx^a} \times \pi_I)|_\Lambda \cl
\Lambda \isoto \Lambda'.
\eneq
\begin{lemma}\label{lem:Hamilt-isot}
We consider a $\Cd^\infty$-map $\Phi\cl \symx \times I\to \symx$ such that
$\phi_t$ is a symplectic isomorphism for each $t\in I$ and we use the above
notations.
Then there exists a Lagrangian submanifold
$\Lambda \subset \symx \times \symx^a \times T^*I$
satisfying~\eqref{eq:Lambda-Lambda'} if and only if $\Phi$ is a Hamiltonian
isotopy.
In this case the possible $\Lambda$ can be written
\eq
\label{eq:def-lambda}
\Lambda & = &
\set{\bl\Phi(v,t), v, t, -f(\Phi(v,t),t)\br}{ v\in\symx, t\in I},
\eneq
where the function $f\cl \symx \times I\to\R$ is defined 
by $(v_\Phi)_t = H_{f_t}$
up to addition of a function depending only on $t$.
\end{lemma}
If $\Lambda$ exists we also have, extending
notation~\eqref{eq:convolution_of_sets} to the case where one manifold is
not necessarily a cotangent bundle:
\eqn
&&\Lambda_{t} = \Lambda\conv T^*_{t}I.
\eneqn
\begin{proof}
We write $T^*I \simeq I\times\R$. A manifold $\Lambda$
satisfying~\eqref{eq:Lambda-Lambda'} is written
$$
\Lambda = \set{\bl\Phi(v,t), v, t, \tau(v,t)\br}{v\in\symx^a, t\in I}
$$
for some function $\tau\cl \symx^a \times I \to \R$.  Let us write down the
condition that $\Lambda$ be Lagrangian.  For a given
$(v,t) \in \symx^a \times I$ and $p= (\Phi(v,t), v, t, \tau(v,t)) \in \Lambda$
the tangent space $T_p\Lambda$ is generated by the vectors
$$
\theta_0 = ( (v_\Phi)_t, 0, 1,
\frac{\partial\tau}{\partial t} )
\quad \mbox{and} \quad
\theta_\nu = ( (d\phi_t)(\nu), \nu,0, (d\tau_t)(\nu) ),
$$
where $\nu$ runs over $T_v\symx^a$.  Since $\phi_t$ is a symplectic
isomorphism the $\theta_\nu$'s are mutually orthogonal for the symplectic
structure of $\symx \times \symx^a \times T^*I$. Hence $\Lambda$ is Lagrangian
if and only if $\theta_0$ and $\theta_\nu$ also are orthogonal, which is
written:
\eqn
0 &=&
\omega\bl (v_\Phi)_t , (d\phi_t)(\nu)\br - (d\tau_t)(\nu) \\
&=&
\bl \h((v_\Phi)_t) - d(\tau_t\circ\phi_t^{-1})\br ((d\phi_t)(\nu)).
\eneqn
This holds for all $\nu \in T_v\symx^a$ if and only if
$\h((v_\Phi)_t)=d(\tau_t\circ\phi_t^{-1})$, or equivalently
$-H_{\tau_t\circ\phi_t^{-1}}=(v_\Phi)_t$.
\end{proof}

\paragraph{Exact case}
We assume that the symplectic form $\omega$ is exact and write $\omega =
d\alpha$. We consider $\Phi\cl \symx \times I\to \symx$ as above but now
we ask that $\phi_t^*(\alpha) =\alpha$ for all $t\in I$.
Then it is well-known (see for example~\cite[Corollary 9.19]{MDS95})
that $\Phi$ is a Hamiltonian isotopy. More precisely 
$v_\Phi$ is the Hamiltonian vector field of
\eq\label{eq:defin_of_f}
f = \langle \alpha, v_\Phi \rangle 
\cl \symx\times I\to \R.
\eneq
Indeed the condition on $\phi_t$ implies by derivation
$\shl_{v_\Phi} (\alpha) = 0$. Hence Cartan's formula
yields:
$$
d(f_t) = \shl_{(v_\Phi)_t} (\alpha)
-  \iota_{(v_\Phi)_t} (\omega) 
= -  \iota_{(v_\Phi)_t} (\omega) 
= -\h((v_\Phi)_t).
$$
This holds in particular when $\symx = \dTM$ for some manifold $M$. We
consider $\Phi\cl \dTM\times I\to \dTM$ such that 
\eq\label{hyp:isot1bis}
&&\begin{cases}
\mbox{$\phi_t$ is a homogeneous symplectic isomorphism for each $t\in I$,} \\
\phi_0 = \id_{\dTM}.
\end{cases}
\eneq
In this case the function $f$ given in~\eqref{eq:defin_of_f} is homogeneous of
degree $1$ in the fibers of $\dTM$ and it is the only homogeneous function
such that $(v_\Phi)_t = H_{f_t}$. So we have the first part of the following
lemma.
\begin{lemma}\label{lem:homog-Hamilt-isot}
Let $\Phi\cl \dTM\times I\to \dTM$ satisfying~\eqref{hyp:isot1}. Then
\bnum
\item 
$\Phi$ is a Hamiltonian isotopy and there exists a unique conic Lagrangian submanifold
$\Lambda$ of $\dTM\times \dTM\times T^*I$
satisfying~\eqref{eq:Lambda-Lambda'}{\rm:} setting
$f = \langle \alpha_M, \partial\Phi/\partial t \rangle$ we have
$$
\Lambda  = 
\set{\bl\Phi(x,\xi,t), (x, -\xi), (t, -f(\Phi(x,\xi,t),t))\br}%
{ (x,\xi)\in \dTM, t\in I},
$$
\item
the set $\Lambda\cup T^*_{M\times M\times I}(M\times M\times I)$ is
closed in $T^*(M\times M\times I)$
and for any $t\in I$ the inclusion $i_t\cl M\times M \to M\times M\times I$ is
non-characteristic for $\Lambda$ and the graph of $\phi_t$ is
$\Lambda_{t} = \Lambda\conv T^*_{t}I$.
\enum
\end{lemma}
\begin{proof}
(i)  is already proved. 

\medskip\noindent
(ii) In local homogeneous symplectic coordinates 
$(x,y;\xi,\eta)\in T^*(M\times M)$, 
$(t;\tau)\in T^*I$, the construction of $\Lambda$
implies that
for any compact set $C\subset M\times M \times I$
there exists $D>0$ such that $|\tau|\leq D\vert\xi\vert$,
$\vert\xi\vert\leq D\vert\eta\vert$ and $\vert\eta\vert\leq D\vert\xi\vert$
for any
$(x,y,t;\xi,\eta,\tau)\in \Lambda \cap \opb{\pi_{M\times M \times I}}(C)$.
Hence the same inequalities hold on the closure $\ol{\Lambda}$ of $\Lambda$.
Hence if $(x,y,t;\xi,\eta,\tau)\in\ol{\Lambda}\setminus
(\dTM\times\dTM\times T^*I)$, then
$\xi=\eta=0$ and $\tau=0$, and hence it belongs to the zero-section.  

Hence $\Lambda\cup T^*_{M\times M\times I}(M\times M\times I)$ is closed.
We also have seen that $\Lambda$ does not meet
$T^*_{M\times M}(M\times M) \times \dT^*I$ which is the
non-characteristicity condition.
\end{proof}

\subsection{Families of conic Lagrangian submanifolds}
Since the results in this section are well-known (they go back to
Paulette Libermann), we state them without proofs.
Note that we only use them in Corollary~\ref{thm:famLagr}.

\begin{definition}\label{def:family-Lagr}
Let $M$ be a manifold and $I$ an open interval containing $0$. 
Let $S_0$ be a closed conic Lagrangian submanifold of $\dTM$.
A deformation of $S_0$ indexed by $I$ is the data of a $\Cd^\infty$-map
$\Psi\cl S_0\times I \to \dTM$ such that, setting
$\psi_t \eqdot \Psi(\cdot,t)$ and $S_t=\psi_t(S_0)$, we have
\bnum
\item $\psi_0$ is the identity embedding,
\item $\psi_t$ is homogeneous for the action of $\R_{>0}$ for each $t\in I$,
\item $S_t$ is a closed conic Lagrangian submanifold of $\dTM$
for each $t\in I$,
\item the map $S_0\times I \to (\dTM) \times I$, $(s,t) \mapsto (\Psi(s,t),t)$,
is an embedding.
\enum
\end{definition}
 We let
$S' = \{ (s,t); \; t\in I, s \in S_t    \} \subset (\dTM)\times I$
be the image of the embedding in (iv). So it is a closed submanifold of
$(\dTM)\times I$.
Note that $\psi_t$ induces a diffeomorphism $\psi_t \cl S_0 \isoto S_t$ for each
$t\in I$.

\begin{lemma}\label{lem:family-Lagr}
Let $S_0$ be a closed conic Lagrangian submanifold of $\dTM$ and
let $\Psi\cl S_0\times I \to \dTM$ be a deformation of $S_0$ as above.
Then there exists a unique closed conic Lagrangian submanifold
$S\subset \dT^*(M\times I)$ such that 
$\dT^*(M\times I)\to T^*M\times I$ induces a diffeomorphism
$S\isoto S'$.
\end{lemma}

Moreover for any $t\in I$ the inclusion $i_t\cl M \to M\times I$ is
non-characteristic for ${S}$  and we have $S_{t} = S\conv T^*_{t}I$.

We remark that $S$, like $S'$, only depends on the family of $\{S_t\}_t$, not
on the parametrization $\Psi$.

For a deformation of a closed conic Lagrangian submanifold we consider a
condition similar to~\eqref{hyp:isot2}.
\eq\label{hyp:famLagr}
&&\parbox{60ex}{
there exists a compact subset $A$ of $M$ such that for all $t\in I$: \\[1ex]
\hs{15ex}$\left\{\ba{l}
\psi_t|_{S_0 \cap \opb{\dot\pi_M}(M\setminus A)}
= \id_{S_0 \cap \opb{\dot\pi_M}(M\setminus A)}, \\[1ex]
\psi_t(S_0 \cap \opb{\dot\pi_M}(A)) \subset \opb{\dot\pi_M}(A).
\ea\right.$}
\eneq

\begin{proposition}\label{prop:extending-isotopies}
Let $S_0$ be a closed conic Lagrangian submanifold of $\dTM$ and
let $\Psi\cl S_0\times I \to \dTM$ be a deformation of $S_0$
satisfying~\eqref{hyp:famLagr}.

Then there exists $\Phi\cl \dTM\times I\to \dTM$ satisfying
hypotheses~\eqref{hyp:isot1} and~\eqref{hyp:isot2} such that
$$
\Phi|_{S_0 \times I} = \Psi.
$$  
\end{proposition}

\subsection{Adding a variable}\label{sec:adding}
In this subsection we recall the link between non-homogeneous symplectic
geometry and homogeneous symplectic geometry with an extra variable.

We denote by $(s,\sigma)$ the coordinates on $T^*\R$ 
with $\sigma ds$ as the Liouville form.
For a manifold $M$  we define the map
\eqn
&&\rho=\rho_M\cl T^*M\times\dT^*\R\to T^*M,
\qquad (x,\xi,s,\sigma) \mapsto (x,\xi/\sigma).
\eneqn

We consider a Hamiltonian isotopy $\Phi\cl T^*M\times I\to T^*M$ as in
Appendix~\ref{app:symp} but we do not assume that
it is homogeneous.
We shall show that $\Phi$ lifts to a homogeneous Hamiltonian isotopy of 
$T^*M\times \dT^*\R$.

Let $f\cl T^*M\times I\to\R$ be a function 
such that $\partial\Phi/\partial t=H_{f_t}$ (see \eqref{def:f}).
We set
\eqn
&&\tf\seteq(f\conv \rho)\cdot \sigma. 
\eneqn
Then $\tf_t$ is a homogeneous 
function on $T^*M\times \dT^*\R$ of degree $1$.

\begin{proposition}\label{prop:homogenisation}
Let $\Phi\cl T^*M\times I\to T^*M$ be a Hamiltonian isotopy
and let $f$ and $\tf$ be as above.
\bnum
\item
Then there exists a homogeneous Hamiltonian isotopy
$\tw \Phi\cl (T^*M\times \dT^*\R)\times I\to T^*M\times \dT^*\R$
such that $\partial\tw\Phi/\partial t=H_{\tf_t}$ and 
the following diagram commutes:
\eqn
&&\xymatrix@C=2cm{
T^*M\times\dT^*\R\times I\ar[r]^{\tw\Phi}\ar[d]_{\rho\times\id_I}
                                       &T^*M\times\dT^*\R\ar[d]_\rho  \\
T^*M\times I\ar[r]^{\Phi}              & T^*M .
}\eneqn
Moreover there exists a $\Cd^\infty$-function $u\cl (T^*M) \times I\to\R$ 
such that
\eq\label{eq:twphi}
&&\tw \Phi(x,\xi,s,\sigma,t) = (x',\xi', s+ u(x,\xi/\sigma,t),\sigma),
\eneq
where $(x',\xi'/\sigma) = \phi_t(x,\xi/\sigma)$.
\item We assume moreover that 
$M$ is connected and 
$\phi_t$ is the identity outside a compact subset $C\subset T^*M$.
Then $\tw \Phi$ extends to a homogeneous Hamiltonian isotopy
$\tw \Phi\cl \dT^*(M\times \R)\times I\to \dT^*(M\times\R)$ such that
\eq\label{eq:twphi2}
&&\tw \Phi(x,\xi,s,0,t) = (x,\xi, s+ v(t),0),
\eneq
for some $\Cd^\infty$-function $v\cl I\to\R$.
\enum
\end{proposition}
\Proof
We have to describe the Hamiltonian vector field $H_{\tw f}$ of $\tw f$.
We denote by $p\cl T^*M\times \dT^*\R \to\dT^*\R$ the projection
$(x,\xi,s,\sigma) \mapsto (s,\sigma)$.
Then $(\rho,p)$ defines an isomorphism
\eq
&&\psi\cl T^*M\times \dT^*\R\isoto T^*M\times \dT^*\R.
\eneq
For a point $q= (x,\xi,s,\sigma)\in  T^*M\times\dT^*\R$,
$\psi$ defines an isomorphism on the tangent spaces:
\eq
\label{eq:dec-tan-space}
&& d\psi=d\rho_q \times dp_q \cl T_q(T^*M\times T^*\R )
\isoto T_{(x,\xi/\sigma)}(T^*M) \oplus T_{(s,\sigma)}(T^*\R).
\eneq
Setting $\omega_M = d\alpha_M$, where $\alpha_M$ is the Liouville form on
$T^*M$, we have
\eqn
\alpha_{M\times \R}|_{T^*M\times\dT^*\R}
&=& \sigma \rho^*(\alpha_M) + p^*(\alpha_\R),  \\
\omega_{M\times \R}|_{T^*M\times\dT^*\R}
&=&  \sigma \rho^*(\omega_M) + p^*(\omega_\R)
+ d\sigma \wedge \rho^*(\alpha_M).
\eneqn
In the sequel we fix $t$ and we set $\tf_t=\tw f(\cdot,t)$
and $f_t= f(\cdot,t)$.
Then $H_{\tf_t}$ is determined by $\iota_{H_{\tf_t}} (\omega_{M\times \R})
= - d\tf_t$. We decompose $(H_{\tf_t})_q = v_M +v_\R$ according
to~\eqref{eq:dec-tan-space} and we also use the decomposition of
$T^*_q(T^*M\times T^*\R )$ induced by~\eqref{eq:dec-tan-space}.
Then we find
$$
\iota_{H_{\tf_t}} (\omega_{M\times \R})
= \bigl(
\sigma \iota_{v_M}(\omega_M) + \langle v_\R, d\sigma \rangle \alpha_M
\bigr)  + \bigl(
 \iota_{v_\R}(\omega_\R) - \langle v_M, \alpha_M \rangle d\sigma
\bigr).
$$
Since $d\tf_t = \sigma \rho^* df_t + \rho^*(f_t) d\sigma$ we obtain
\eqn
- df_t&=& \iota_{v_M}(\omega_M)
+ \sigma^{-1} \langle v_\R, d\sigma \rangle \alpha_M,\\
- \rho^*(f_t) d\sigma &=& \iota_{v_\R}(\omega_\R) 
- \langle v_M, \alpha_M \rangle d\sigma.
\eneqn
The second equality gives $v_\R = a\, \frac{\partial}{\partial s}$ for some
function $a$. Then we have $\langle v_\R, d\sigma \rangle =0$, which implies
$v_M = H_{f_t}$ by the first equality, and hence
$a= (f_t - \langle H_{f_t}, \alpha_M \rangle) \circ \rho
= (f_t - \eu_M(f_t))\circ \rho$. Finally, letting
$g\seteq f-\eu_M(f)$ be a function on $T^*M\times I$,
we obtain
$$
\psi_*(H_{\tf_t}) = H_{f_t}
+ \rho^*(g_t)\frac{\partial}{\partial s} .
$$
Let us define $u\cl T^*M\times I\to \R$ by the differential equation:
\eq
&&\left\{\ba{rl}
\dfrac{\partial u}{\partial t}&=g_t\circ \phi_t,\\[1ex]
u\vert_{t=0}&=0.
\ea\right.
\eneq
We define $\tw \Phi$ by \eqref{eq:twphi}.
Then, we can see easily that 
$$
\dfrac{\partial\tw\Phi}{\partial t}=H_{\tf_t}.
$$
Hence $\tw\Phi$ is the desired homogeneous Hamiltonian isotopy.

\bigskip
\noindent
(ii)\quad The functions $f_t$ and $g_t$ are constant functions outside $C$.
Hence $u_t$ is also a constant function outside $C$ taking the value
$v(t)$.
Then 
$\tw\Phi$ extends
to a homogeneous Hamiltonian isotopy
$\tw \Phi\cl \dT^*(M\times \R)\times I\to \dT^*(M\times\R)$
by \eqref{eq:twphi2}.
\QED

\providecommand{\bysame}{\leavevmode\hbox to3em{\hrulefill}\thinspace}

\vspace*{1cm}
\noindent
\parbox[t]{21em}
{\scriptsize{
\noindent
St{\'e}phane Guillermou\\
Institut Fourier, Universit{\'e} de Grenoble I, \\
email: Stephane.Guillermou@ujf-grenoble.fr\\

\medskip\noindent
Masaki Kashiwara\\
Research Institute for Mathematical Sciences, Kyoto University \\  
e-mail: masaki@kurims.kyoto-u.ac.jp

\medskip\noindent
Pierre Schapira\\
Institut de Math{\'e}matiques,
Universit{\'e} Pierre et Marie Curie\\
e-mail: schapira@math.jussieu.fr\\
}}

\end{document}